\documentclass[11pt]{article}
\usepackage{amssymb} 
\usepackage{mathrsfs}
\usepackage{latexsym,amsmath,amssymb,amsfonts,epsfig,graphicx,cite,psfrag}
\usepackage{eepic,color,colordvi,amscd}
\definecolor{blue}{rgb}{0,0,0.9} 
\definecolor{red}{rgb}{0.9,0,0} 
\definecolor{green}{rgb}{0,0.9,0}

\usepackage{ebezier} 
\usepackage{verbatim}
\usepackage{subfigure}
\usepackage{enumitem}
\usepackage{algorithm}
\usepackage{algorithmic}
\usepackage{amsthm}
\usepackage{comment}
\usepackage{blkarray}
\usepackage{hyperref}
\usepackage{arydshln} 
\usepackage{tikz}
\usepackage{color}
\usepackage{fancyhdr}
\usepackage{lipsum} 
\theoremstyle{plain}
\newtheorem{theo}{Theorem}[section]
\newtheorem{claim}[theo]{Claim}
\newtheorem{ques}[theo]{Question}
\newtheorem{lem}[theo]{Lemma}
\newtheorem{coro}[theo]{Corollary}

\newtheorem{prop}[theo]{Proposition}

\theoremstyle{definition}
\newtheorem{defi}[theo]{Definition}

\theoremstyle{remark}
\newtheorem{rem}[theo]{Remark}
\def\<{\big\langle}
\def\>{\big\rangle}
\def\E{\mathcal{E}}

\def\B{\mathcal{B}}
\def\V{\mathcal{V}}
\def\C{\mathcal{C}}

\def\I{\mathcal{I}}
\def\J{\mathcal{J}}
\def\R{\mathbb{R}}

\def\tU{\triangle U}

\def\N{\mathbb{N}}

\def\S{\mathbb{S}}

\def\rr{{\rm rank}}

\def\dd{{\rm diag}}

\def\P{{\bf proj}}

\def\1{{\bf 1}}

\def\O{\mathcal{O}}

\def\tw{{\rm tw}}
\def\pw{{\rm pw}}
\def\setm{[m]\sqcup\{0\}}
\def\({\left(}
\def\){\right)}

\tikzset{every picture/.style={line width=0.75pt}} 
\let\svthefootnote\thefootnote
\newcommand\blankfootnote[1]{%
	\let\thefootnote\relax\footnotetext{#1}%
	\let\thefootnote\svthefootnote%
}

\begin{document}
	\title{Exploring chordal sparsity in semidefinite programming with sparse plus low-rank data matrices}
	\author{Tianyun Tang
\thanks{Institute of 
Operations Research and Analytics, National
         University of Singapore, Singapore
         117602 ({\tt ttang@u.nus.edu}).
         }, \quad 
	 Kim-Chuan Toh\thanks{Department of Mathematics, and Institute of 
Operations Research and Analytics, National
         University of Singapore, 
       Singapore
         119076 ({\tt mattohkc@nus.edu.sg}).  The research of this author is supported
by the Ministry of Education, Singapore, under its Academic Research Fund Tier 3 grant call (MOE-2019-T3-1-010).}
	 }
	\date{\today} 
	\maketitle

\begin{abstract}
Semidefinite programming (SDP) problems are challenging {to solve} because of their {high} dimensionality. However, {solving} sparse SDP problems with small tree-width are known to be relatively {easier} because: {(1)} they can be decomposed into smaller multi-block SDP problems through chordal conversion; {(2)} they have low-rank optimal solutions. In this paper, we study more general SDP problems whose coefficient matrices have sparse plus low-rank (SPLR) structure. We develop a unified framework to convert such problems into sparse SDP problems with bounded tree-width. Based on this, we derive rank bounds for SDP problems with SPLR structure, which are tight in the worst case. 
\end{abstract}	
	
\bigskip
\noindent{\bf keywords:} semidefinte programming, sparsity pattern, low-rank property
\\[5pt]
	

\section{Introduction}

\subsection{Linear SDP problem}
Let $\mathbb{S}^n$ be the vector space of $n\times n$ symmetric matrices endowed with the standard trace inner product and its induced Frobenious norm. 
We consider the following linear semidefinite programming (SDP) problem,

\begin{equation}\label{SDP} 
\min\left\{ \<A_0,X\>:\ \forall i\in [m]\ b_i^l\leq \<A_i,X\>\leq b_i^u,\ X\in \S^n_+ \right\}, \tag{SDP} 
\end{equation} 
where for any $i\in \setm,$ $A_i\in \S^n$ and for any $i\in [m],$ $-\infty\leq b_i^l\leq b_i^u< +\infty.$ 
In the above, $\mathbb{S}^n_+$ denotes the cone of positive semidefinite matrices in $\mathbb{S}^n.$
The problem (\ref{SDP}) is a traditional SDP problem, which can be solved by many well-developed solvers like SeDuMi \cite{sturm1999using}, SDPT3 \cite{TTT,T3Q}, and SDPNAL+ \cite{SDPNALp2,SDPNALp1}. Those solvers would run into memory issue when $n$ is large because the dimensionality of the primal variable $X$ is $\Omega(n^2).$ {Therefore}, in practice, dimension reduction techniques are usually applied when solving SDP problems with huge matrix dimension $n.$ {In this paper, we aim to develop a new technique to tackle SDP problems with sparse plus low-rank structure in the data matrices as we shell elaborate in the next few subsections.}

\subsection{Sparse SDP and chordal conversion}
For (\ref{SDP}), we call {a} graph\footnote{In this paper, by graph, we always mean simple undirected graph.} $G=([n],\E)$ {as} its sparsity pattern if for any $i\in \setm,$ $A_i\in \S^n(\E,0),$ where $\S^n(\E,0)$ is defined as follows:
\begin{equation}
\S^n(\E,0):= \left\{ X\in \S^n:\ X_{ij}=X_{ji}=0,\ \forall i,j\in [n],\ {\rm s.t.}\ i\neq j,\ ij\notin \E \right\}. \notag
\end{equation} 
We use $\P_{\E}(\cdot)$ to denote the orthogonal projection mapping from $\mathbb{S}^n$ onto $\S^n(\E,0)$. {In the following, we will use some definitions in graph theory such as tree decomposition: $(T,\V),$ width: ${\rm wid}(T),$ tree-width: $\tw(G)$, and path-width: $\pw(G).$ Readers may 
refer to Section~\ref{Sec-twpw} first for these definitions.}

One important dimension reduction technique for sparse SDP is chordal conversion, introduced by Kojima et al. in \cite{nakata2003exploiting,fukuda2001exploiting}. According to Grone’s and Agler’s theorems \cite{grone1984positive,agler1988positive}, the problem (\ref{SDP}) can be converted into a multi-block SDP problem with the matrix dimension of each block bounded by $\tw(G)+1$ {through} a process called chordal conversion. When {the tree-width} $\tw(G)$ is small, the {matrix} dimension is significantly reduced so that the new problem can be solved by interior point method{s} or first order method{s} \cite{zheng2020chordal,kang2024fast} efficiently. 
Although sparse SDP problems in general may not have small tree-width \cite{do2024note}, many graphs {arising} from domains like optimal power flow \cite{zhang2023parameterized}, 
control theory \cite{fazelnia2016convex}, and polynomial optimization \cite{madani2014rank}, {often} have constant tree-width. 
In this case, Zhang and Lavaei \cite{zhang2021sparse} show that when the matrices $A_i$'s satisfy {a certain} partial separability 
{property}, (\ref{SDP}) can be solved in nearly linear time {with respect to the matrix dimension $n$}.

\subsection{SDP with sparse plus low-rank structure}\label{subsec-example} 
In \cite[section 14.1]{vandenberghe2015chordal}, Vandenberghe et al. described a kind of SDP problems with dense data matrices whose sparsity can be improved through variable transformation, a technique first introduced in 
\cite[section 6]{fukuda2001exploiting}. The key structure is that the data matrices, although dense, have sparse plus low-rank decomposition. We formally state the definition of the sparse plus low-rank structure as follows{.} 

\begin{defi}\label{SPLRdefi}
Consider $\ell\in \N$ and graph $G=([n],\E).$ We say that the problem (\ref{SDP}) has $(G,\ell)-$SPLR {(sparse plus low-rank)} structure if the {\bf there exists} a decomposition of its coefficient matrices $A_i=A_i^s+A_i^l$ such that $A_i^s\in \S^n(\E,0)$ for any $i\in \setm$ and $\rr\(\left[A_0^l,A_1^l,\ldots,A_m^l\right]\)\leq \ell.$ 
\end{defi}

The SPLR structure {often} appears in many SDP problems such as minimum bisection SDP \cite{rendl1999semidefinite}, Lov\'asz theta SDP \cite{Theta}, quadratic knapsack SDP \cite{helmberg2000semidefinite} and maximum variance unfolding problem \cite{weinberger2006unsupervised}. More generally, consider the following SDP relaxation of {a} general non-convex quadratic programming problem \cite{burer2009copositive}:
\begin{equation}\label{BQP}
\min\left\{ x^\top Qx+2c^\top x:\  Ax=b,\ x_i\in \{0,1\},\ \forall i\in B,\ x\in \R^n_+\right\},
\end{equation}
where $Q\in \S^n,$ $c\in \R^n,$\ $A\in \R^{\ell\times n},$ $b\in \R^\ell$ and $B\subset [n].$
The SDP relaxation of (\ref{BQP}) \cite[section 5]{burer2009copositive} {has the following form}:
\begin{multline}\label{BQPSDP}
\min\Bigg\{ \left\langle\begin{bmatrix} 0&c^\top \\ c&Q \end{bmatrix},\ Y\right\rangle:\ Y_{11}=1,\ Y_{2:n+1,1}\in \R^n_+,\ Y_{(i+1)(i+1)}=Y_{(i+1)1}\ \forall i\in B,\\
 \left\langle \begin{bmatrix}-b^\top \\ A^\top\end{bmatrix}\left[ -b,A \right],Y \right\rangle=0,\ Y\in \S^{n+1}_+\Bigg\}.
\end{multline}
Let $G=\([n+1]\setminus\{1\},\E\)$ be the sparsity pattern of $Q.$ If we consider {the data matrices of the} first three constraints in (\ref{BQPSDP}) as sparse and the fourth one as low-rank, then Problem (\ref{BQPSDP}) has $\(G',\ell\)-$SPLR structure where $G'=\([n+1],\ \E':=\E\sqcup\{1(i+1):\ i\in [n]\}\).$

One popular technique to improve the sparsity of (\ref{SDP}) with $(G,\ell)-$SPLR structure is {\bf congruent transformation}, a method first introduced by Kojima et al. in \cite[section 6]{fukuda2001exploiting}. In detail, consider variable {transformation} $Y:=P^{-\top} X P^{-1}$ for some {given} invertible matrix $P\in \R^{n\times n}.$ The problem (\ref{SDP}) is equivalent to the following problem,

\begin{equation}\label{SDPcong}
\min\left\{ \<PA_0P^\top,Y\>:\ \forall i\in [m]\ b_i^l\leq \<PA_iP^\top,Y\>\leq b_i^u,\ Y\in \S^n_+ \right\}.
\end{equation}
Suppose $W\in \R^{n\times \ell}$ is {an orthonormal} basis matrix of $\left[A_0^l,A_1^l,\ldots,A_m^l\right]$ and $P:=[W,V]^\top,$ where $V\in \R^{n\times (n-\ell)}$ is the orthonormal complement of $W.$ Then for any $i\in \setm,$ $PA_iP^\top=P\(A_i^s+A_i^l\)P^\top=PA_i^sP^\top+PA_i^lP^\top,$ where $PA_i^lP^\top$ is a sparse matrix with the nonzero entries appearing in the first $\ell\times \ell$ principle submatrix. If the orthonormal matrix {$P$} is sparse, then the matrix $PA_i^sP^\top$ could be also sparse for any $i\in \setm$ so that the problem (\ref{SDPcong}) has a sparse sparsity pattern. 

Although congruent transformation {can be} an effective preprocessing technique in practice, this technique remains a heuristic {that has only been applied successfully to highly specialized SPLR structure such as the SDP relaxation of minimum bisection problems which has a constraint matrix of the form $ee^\top$, where $e\in\mathbb{R}^n$ is the matrix of all ones.} To our {best} knowledge, there is currently no systematic way to select $P$ so that the new problem (\ref{SDPcong}) is sparse. Moreover, the impact of the congruent transformation on the tree-width of the underlying graph has not been studied yet. In other words, even if the transformed problem (\ref{SDPcong}) is sparse, its tree-width could be much greater than the original $\tw(G).$ Considering these two facts, it would be natural to ask the following question:

\begin{ques}\label{quesp}
Is there a systematic way to convert (\ref{SDP}) with $(G,\ell)-$SPLR structure into a sparse SDP problem whose tree-width is bounded by some function $f(\tw(G),\ell)$? 
\end{ques}

In this paper, we will answer the Question~\ref{quesp} affirmatively. Our contributions are stated in the following two subsections. Before that, we present notation that will be used frequently in this paper.

{
{\bf Notation.}
We use the notation $Y_{I,J}$ to denote the submatrix obtained 
by extracting the rows and columns of a matrix $Y$ in the index sets $I$ and $J$, respectively.
We also use $Y_{I,:}$ to denote the submatrix obtained from $Y$ by extracting the 
rows in the index set $I$. 
Also, we use a notation such as ``$\ell+1:2\ell$" to denote the index set $\{\ell+1,\ell+2,\ldots,2\ell\}$.
For any two matrices $P,Q$ with compatible dimensions, we use $[P,Q]$ and $[P;Q]$ to denote the concatenation of $P$ and $Q$ column-wise and row-wise respectively.
}

\subsection{Contribution 1: sparse extension}

In Section~\ref{Sec:SET}, we will propose a systematic method called ``sparse extension" to transfer (\ref{SDP}) with SPLR structure into a sparse SDP problem with bounded tree-width. Instead of using congruent transformation, the idea of ``sparse extension" is to introduce auxiliary variables to extend (\ref{SDP}) into a high{er} dimensional space where the low-rank data matrices are decomposed into several sparse matrices. In order to state our contribution as a theorem, we {need} the following definition.

\begin{defi}\label{defi:ex}
Consider the following SDP problems 
\begin{equation}\label{SDPex} 
\min\left\{ \big\langle\widehat{A}_0,\widehat{X}\big\rangle:\ \forall i\in [\hat{m}],\ \widehat{b}_i^l\leq \big\langle\widehat{A}_i,\widehat{X}\big\rangle\leq \widehat{b}_i^u,\ \widehat{X}\in \S^{\hat{n}}_+ \right\}, \tag{SDPex} 
\end{equation} 
where for any $i\in \{0\}\sqcup[\hat{m}],$ $\hat{A}_i\in \S^{\hat{n}}$ and for any $i\in [\hat{m}],$ $-\infty\leq \widehat{b}_i^l\leq \widehat{b}_i^u< +\infty.$ We call (\ref{SDPex}) {an} {\bf extension} of (\ref{SDP}) if there exists {an} index set $\I{\subset} [\hat{n}]$ such that for any feasible $\widehat{X}$ of (\ref{SDPex}), $\widehat{X}_{\I,\I}$ is feasible for (\ref{SDP}) and {$\langle\widehat{A}_0,\widehat{X}\rangle=\langle A_0,\widehat{X}_{\I,\I}\rangle.$} Also, for any feasible $X$ for (\ref{SDP}), there exists feasible solution $\widehat{X}$ for (\ref{SDPex}) such that $\widehat{X}_{\I,\I}=X.$
\end{defi}

From Definition~\ref{defi:ex}, it is easy to see that (\ref{SDPex}) and (\ref{SDP}) are equivalent in the sense that they have the same optimal value. Also, suppose $\widehat{X}$ is {an} optimal solution of (\ref{SDPex}), then $\widehat{X}_{\I,\I}$ is {an} optimal solution of (\ref{SDP}). With Definition~\ref{defi:ex}, the property of our sparse extension can be summarized {in} the following theorem.

\begin{theo}\label{theo:spex}
Suppose the problem (\ref{SDP}) has $(G,\ell)-$SPLR structure for some $\ell\in \N$. {Let $\(T,\V\)$ be a tree-decomposition of $G$ such that $|T|=p$ and its width 
{\rm wid}($T,{\cal V}$)
=\tw(G).} Then it has a sparse extension (\ref{SDPex}), {whose precise form is given 
in \eqref{newprob},}
with sparsity pattern $\widehat{G}=(\widehat{n},\widehat{E})$ such that $\widehat{n}\leq n+{2p\cdot \ell}$ and $\tw(\widehat{G})\leq \tw(G)+3\ell.$ Moreover, if {T is a path and $\tw(G)=\pw(G)$ (path-width),} then the bound{s} can be strengthened to $\hat{n}\leq n+{p\cdot\ell}$ and $\tw(\widehat{G})\leq \tw(G)+2\ell.$
\end{theo}

Theorem~\ref{theo:spex} answers the Question~\ref{quesp} affirmatively by choosing $f(\tw(G),\ell)=\tw(G)+3\ell.$ This implies that an SDP problem with $(G,\ell)-$SPLR 
can be solved efficiently {(through chordal conversion of its sparse extension)} when both $\tw(G)$ and $\ell$ are small.


\subsection{Contribution 2: rank bound}\label{Intro-rank}
Apart from chordal conversion, another popular dimension reduction {technique} is low-rank decomposition, that is, factorizing $X$ into $RR^\top,$ where $R\in \R^{n\times r}$ for some $r\ll n.$ {In this case, (\ref{SDP}) is converted to the following
form:}

\begin{equation}\label{SDPLR} 
\min\left\{ \<A_0,RR^\top\>:\ \forall i\in [m],\ b_i^l\leq \<A_i,RR^\top\>\leq b_i^u,\ R\in \R^{n\times r} \right\}. \tag{SDPLR} 
\end{equation} 
This idea was first proposed by Burer and Monteiro in \cite{BM1,BM2} two decades ago and {has received} significant {attention} \cite{Boumal2,cifuentes2021on,tang2024feasible,tang2024solving,lee2022escaping,monteiro2024low} {since then}. 

Because the problem (\ref{SDPLR}) is only equivalent to (\ref{SDP}) when the latter has an optimal solution of rank at most $r,$ the rank bound of (\ref{SDP}) is an important research {topic}.  The most famous rank bound for (\ref{SDP}) is the Barvinok-Pataki bound \cite{barvinok1995problems,friedland1976subspaces,pataki1998on}, which states that (\ref{SDP}) has an optimal solution (if exists) of rank at most $\left\lfloor(\sqrt{8m+1}-1)/2\right\rfloor$. Despite of its wide usage, the Barvinok-Pataki bound is only related to the number of constraints while {ignoring} the structure of the data matrices. On the other hand, various data structures have been utilized to establish rank bounds for specific SDP problems  \cite{laurent2014new,madani2017finding,tang2024feasiblesqk}. One notable result is the Dancis theorem \cite{dancis1992positive}, which says that (\ref{SDP}) has an optimal solution of rank at most $\tw(G)+1$, which is smaller than the Barvinok-Pataki bound when $\tw(G)=\O(1)$. However, the Dancis theorem is not {useful} for SDP problems with SPLR-structure, {for which it will give the trivial bound of $n$} because of the {presence of} dense low-rank data matrices.

In Section~\ref{Sec:rank}, we will derive {new} rank bounds for SDP problems with SPLR structure. In order to state our result, we need the following technical definition.

\begin{defi}\label{cvxnb}
Define the mapping $\varphi:\N\rightarrow \N$ such that $\varphi(0):=0$ and for any $\ell\in \N,$ $\varphi(\ell)$ is the {\bf smallest} integer such that for any $U\in \R^{2\ell\times \ell}$ and $M\in \S^\ell,$ the following set is either empty or has an element of rank at most 
$\varphi(\ell):$  
\begin{equation}\label{cvxset}
\left\{ Y\in \S^{2\ell}_+:\ U^\top Y U=M,\ Y_{1:\ell,1:\ell}=I_\ell,\ Y_{\ell+1:2\ell,\ell+1:2\ell}=I_\ell \right\}.
\end{equation}
\end{defi}

\noindent{The structure (\ref{cvxset}) will appear in a rank reduction procedure in our later theoretical analysis. 
Now, we state our {new} rank bounds based on the tree-width {in the following theorem whose proof will be presented in Subsection~\ref{Subsec:pfthm}.}

\begin{theo}\label{ubthmtree}
Suppose the problem (\ref{SDP}) has $(G,\ell)-$SPLR structure for some $\ell\in \N$.
\begin{itemize}
\item [(i)] The problem (\ref{SDP}) has an optimal solution $X$ (if exists) such that $\rr(X)\leq \tw(G)+\varphi(\ell)+1.$
\item [(ii)] If $\tw(G)=\pw(G),$ then the problem (\ref{SDP}) has an optimal solution $X$ (if exists) such that $\rr(X)\leq \tw(G)+\ell+1.$
\end{itemize} 
\end{theo}

The first rank bound in Theorem~\ref{ubthmtree} is not explicitly given 
in terms of $\ell$ due to the absence of a direct formula for $\varphi(\ell)$. Nonetheless, we can obtain {an} estimate of $\varphi(\ell)$ to achieve tightness within a factor of $\sqrt{3}$, {as stated in the following proposition with its proof presented in Subsection~\ref{Subsec:pfphi}.}

\begin{prop}\label{propphi}
For any $\ell\in \N^+,$ $\ell+1\leq \varphi(\ell)\leq \left\lfloor\(\sqrt{12\ell(\ell+1)+1}-1\)/2 \right\rfloor.$
\end{prop}

{The following theorem shows the tightness of the bounds in Theorem~\ref{ubthmtree}, whose proof will be presented in Subsection~\ref{Subsec:pflb}.}

\begin{theo}\label{lbthmtree} 
{Consider} any $\omega,\ell,n\in \N^+.$
\begin{itemize}
\item [(i)] If $n\geq \omega+\ell+1,$ then there exists a problem (\ref{SDP}) with $(G,\ell)-$SPLR structure such that $\pw(G)=\tw(G)=\omega$ and the rank of any optimal solution is at least $\tw(G)+\ell+1$.
\item [(ii)] If $\omega\geq 3\ell-1$ and $n\geq \omega+3\ell+1$, then there exists a problem (\ref{SDP}) with $(G,\ell)-$SPLR structure such that $\tw(G)=\omega$ and the rank of any optimal solution is at least $\tw(G)+\varphi(\ell)+1$.
\end{itemize}
\end{theo}


\subsection{Outline of the paper}
In Section~\ref{Sec-twpw}, we will introduce some essential background material
on the tree-width and path-width of graphs. In Section~\ref{Sec:SET}, we will present the detailed procedure of our sparse extension and prove Theorem~\ref{theo:spex}. In Section~\ref{Sec:rank}, we will prove the theorems about the rank bounds stated in Subsection~\ref{Intro-rank}. In Section~\ref{Sec-conc}, we give a brief conclusion. In Appendix~\ref{app-illus}, we present some figures to illustrate some graph operations used {in the sparse extension of Section~\ref{Sec:SET}}. In Appendix~\ref{app-useful} and Appendix~\ref{app-useful-1}, we state and prove some results that are used in our theoretical analysis {in Section~\ref{Sec:rank}}.


\section{Preliminaries on tree-width, path-width and binary tree}\label{Sec-twpw}
Consider $G=\([n],\E\).$ $G$ is called a chordal graph if every cycle of more than 3 vertices contains a chord. A clique of $G$ is a complete subgraph. A maximal clique of $G$ is a clique that is not properly contained in any other clique. See Figure~\ref{fig:chordal graph} for an example of {a} chordal graph and its maximal cliques. We will use the following definition of tree decomposition.

\begin{defi}\label{treedefi}\cite[section 12.3]{diestel2005graph}
Consider a graph $G=([n],\E).$ Let $T$ be a tree and $\V=(V_t)_{t\in T}$ be a family of vertex sets $V_t\subset [n]$ indexed by the vertices of $T.$ The pair $(T,\V)$ is called a tree decomposition of $G$ if it satisfies the following three conditions:
\begin{itemize}
\item [(i)] (Vertex cover) $[n]=\cup_{t\in T}V_t$.
\item [(ii)] (Edge cover) For any $ij\in \E,$ there exists $t\in T$ such that $i,j\in V_t.$
\item [(iii)] (Running intersection) Suppose $t_1,t_2,t_3\in T$ and $t_2$ lies on the unique path connecting $t_1$ and $t_3,$ then $V_{t_1}\cap V_{t_3}\subset V_{t_2}.$
\end{itemize}
The width of the tree decomposition {${\rm wid}(T,\V)$} is equal to $\max_{t\in T}\{|V_t|\}-1.$ The {{\em tree-width}} $\tw(G)$ of $G$ is the minimum width among all its tree decompositions. If for all $t\in T,$ $V_t$ is a clique in $G,$ we call $(T,\V)$ a {\bf clique tree decomposition}. 
\end{defi}

\begin{figure}[htb]       
\centering
\tikzset{every picture/.style={line width=0.75pt}} 

\begin{tikzpicture}[x=0.75pt,y=0.75pt,yscale=-1,xscale=1]

\draw   (122.47,203.02) -- (249.42,203.02) -- (249.42,260.08) -- (122.47,260.08) -- cycle ;
\draw   (153.93,174.07) -- (185.38,203.02) -- (122.47,203.02) -- cycle ;
\draw   (217.97,174.07) -- (249.42,203.02) -- (185.38,203.02) -- cycle ;
\draw   (281.43,231.64) -- (249.4,260.08) -- (249.4,203.02) -- cycle ;
\draw   (90.02,231.05) -- (122.47,203.02) -- (122.47,260.08) -- cycle ;
\draw   (154.16,288.81) -- (122.47,260.08) -- (185.38,260.08) -- cycle ;
\draw   (217.07,288.39) -- (185.38,260.08) -- (249.42,260.08) -- cycle ;
\draw    (122.47,203.02) -- (185.38,260.08) ;
\draw    (122.47,260.08) -- (185.38,203.02) ;
\draw    (185.38,203.02) -- (249.42,260.08) ;
\draw    (185.38,259.67) -- (249.4,203.22) ;
\draw    (185.38,203.02) -- (185.38,259.67) ;

\draw (120.4,187.69) node [anchor=north west][inner sep=0.75pt]   [align=left] {{\tiny 1}};
\draw (183.99,187.69) node [anchor=north west][inner sep=0.75pt]   [align=left] {{\tiny 2}};
\draw (248.03,187.69) node [anchor=north west][inner sep=0.75pt]   [align=left] {{\tiny 3}};
\draw (247.58,269.62) node [anchor=north west][inner sep=0.75pt]   [align=left] {{\tiny 4}};
\draw (183.99,270.44) node [anchor=north west][inner sep=0.75pt]   [align=left] {{\tiny 5}};
\draw (120.85,269.62) node [anchor=north west][inner sep=0.75pt]   [align=left] {{\tiny 6}};
\draw (80.26,228.86) node [anchor=north west][inner sep=0.75pt]   [align=left] {{\tiny 7}};
\draw (151.52,162.42) node [anchor=north west][inner sep=0.75pt]   [align=left] {{\tiny 8}};
\draw (217.36,162.42) node [anchor=north west][inner sep=0.75pt]   [align=left] {{\tiny 9}};
\draw (286.82,228.86) node [anchor=north west][inner sep=0.75pt]   [align=left] {{\tiny 10}};
\draw (212.01,292.71) node [anchor=north west][inner sep=0.75pt]   [align=left] {{\tiny 11}};
\draw (147.52,292.71) node [anchor=north west][inner sep=0.75pt]   [align=left] {{\tiny 12}};
\end{tikzpicture}
\caption{Chordal graph $G=([12],\E)$ with $\tw(G)=3$. The maximal cliques are $\{1,6,7\},\{1,2,8\},\{2,3,9\},\{3,4,10\},\{4,5,11\},\{5,6,12\},\{1,2,5,6\},\{2,3,4,5\}.$}
\label{fig:chordal graph}
\end{figure}
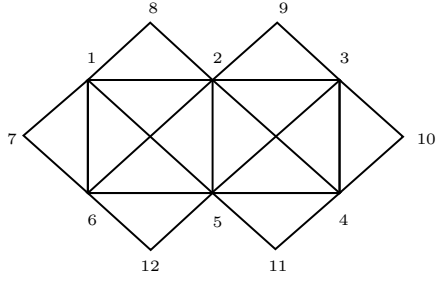

\begin{rem}\label{treerem}
Here we list some results concerning tree decomposition and tree-width. Readers can find them in subsections 12.3 and 12.4 of \cite{diestel2005graph} and section 3 of \cite{vandenberghe2015chordal}.
\begin{itemize}
\item [(1)] A graph $G$ is a chordal graph if and only if it has a clique tree decomposition.
\item [(2)] If $G$ is a chordal graph, then $\tw(G)$ is the size of its maximum clique minus 1.
\item [(3)] For a chordal graph $G=([n],\E),$ its {\bf clique tree} \cite[section 4]{blair1993introduction} is a clique tree decomposition $(T,\V)$ such that $\V$ is the set of maximal cliques. The width of its clique tree is equal to $\tw(G)$. See Figure~\ref{fig:clique tree} for the clique tree of the chordal graph in Figure~\ref{fig:chordal graph}.
\item [(4)] For any graph $G=([n],\E),$ there exists a chordal completion $G'=\([n],\E'\)$ such that $G'$ is a chordal graph, $\E\subset \E'$ and $\tw\(G'\)=\tw\(G\).$
\end{itemize}
\end{rem}

The definitions of path decomposition and path-width \cite{robertson1983graph} are as follows.

\begin{defi}\label{pathdefi}
Consider a graph $G=([n],\E).$ A tree decomposition $(T,\V:=\{V_t:\ t\in T\})$ is called a path decomposition if $T$ is a path. The {{\em path-width}} $\pw(G)$ of $G$ is the minimum width among all its path decompositions. If for any $t\in T,$
 $V_t$ is a clique in $G,$ we call $(T,\V)$ a {\bf clique path decomposition}.
\end{defi}

\begin{rem}\label{pathrem}
Because path decomposition is a special case of tree decomposition, we have that $\pw(G)\geq \tw(G).$ Similar to the tree-width, for any graph $G=([n],\E),$ there exists a chordal completion $G'=([n],\E')$ such that $\E\subset \E'$ and $\pw(G)=\pw(G').$
\end{rem}

\begin{figure}
\centering
\tikzset{every picture/.style={line width=0.75pt}} 

\begin{tikzpicture}[x=0.75pt,y=0.75pt,yscale=-1,xscale=1]

\draw    (116,190.25) -- (132.5,190.25) ;
\draw    (182,189.75) -- (207,190.25) ;
\draw    (260.5,189.75) -- (280.5,189.75) ;
\draw    (147,160.75) -- (147,178.75) ;
\draw    (230,160.75) -- (230,180.75) ;
\draw    (147,196.25) -- (147,214.25) ;
\draw    (230.5,196.25) -- (230.5,214.75) ;

\draw (129,179.5) node [anchor=north west][inner sep=0.75pt]   [align=left] {\{1,2,5,6\}};
\draw (204.5,179.5) node [anchor=north west][inner sep=0.75pt]   [align=left] {\{2,3,4,5\}};
\draw (129,143.5) node [anchor=north west][inner sep=0.75pt]   [align=left] {\{1,2,8\}};
\draw (73.5,180.5) node [anchor=north west][inner sep=0.75pt]   [align=left] {\{1,6,7\}};
\draw (130,214) node [anchor=north west][inner sep=0.75pt]   [align=left] {\{5,6,12\}};
\draw (211,214.5) node [anchor=north west][inner sep=0.75pt]   [align=left] {\{4,5,11\}};
\draw (277.5,180) node [anchor=north west][inner sep=0.75pt]   [align=left] {\{3,4,10\}};
\draw (214,143) node [anchor=north west][inner sep=0.75pt]   [align=left] {\{2,3,9\}};

\end{tikzpicture}
\caption{Clique tree of the chordal graph in Figure~\ref{fig:chordal graph}.}
\label{fig:clique tree}
\end{figure}
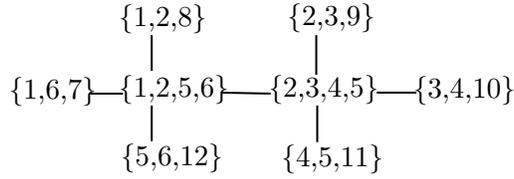

We now introduce a special tree decomposition that will be used in our proof of Theorem~\ref{ubthmtree}.

\begin{defi}\label{combdefi}
Consider a graph $G=([n],\E){.}$ {We} call a tree decomposition $(T,\V)$ {of $G$} a {\bf binary tree decomposition} if the maximum degree of $T$ is at most $3.$ Furthermore, we call it a {\bf binary clique tree decomposition} if it is {also} a clique tree decomposition.
\end{defi}

We call such a tree decomposition $(T,\mathcal{V})$ binary because a tree $T$ can be partially ordered as a rooted binary tree if and only if its maximum degree is at most 3. Such a partial ordering can be obtained by first designating a vertex $z\in T$ of degree less than 3 as the root.
For any edge $xy$ in the edge set $\E(T)$ of the tree decomposition, we call $x$ a child of $y$ {(and $y$ the parent of $x$)} if $y$ is contained in the unique path from $x$ to $z$. Because the degree of the root $z$ is less than 3, every vertex $x\in T$ of degree 3 is not the root and has a parent. Thus, $x$ has at most two children, so $T$ is a binary tree.

Now we will show that each graph $G=\([n],\E\)$ has a binary tree decomposition whose width equals $\tw(G)$. We first define an operation on its tree decomposition that won't influence its width but reduce the number of high-degree vertices in its tree decomposition. 

\begin{defi}\label{splitop}
Consider a graph $G=([n],\E)$ and a corresponding tree decomposition $\(T,\V:=\{V_t:\ t\in T\}\).$ For any $x\in T$ with degree $k\geq 4,$ let its neighbourhood be $\{y_1,y_2,\ldots,y_k\}\subset T.$ We define a new tree decomposition\footnote{We will show that it is indeed a tree decomposition later.} $\(T',\V':=\{V'_t:\ t\in T'\}\)$ of $G,$ which is obtained by splitting $x$ into a path of $k$ vertices. In detail, $T'$ has vertex set $\(T\setminus\{x\}\)\sqcup\left\{ x_1,x_2,\ldots,x_k \right\}$ and edge set $$\E\(T'\):=\(\E(T)\setminus\{xy_1,xy_2,\ldots,xy_k\}\) \sqcup \left\{ x_1y_1,x_2y_2,\ldots,x_ky_k,x_1x_2,x_2x_3,\ldots,x_{k-1}x_k \right\}.$$ Moreover, for any $i\in [k],$ let $V'_{x_i}:=V_x.$ For any $y\in T'\setminus\left\{ x_1,x_2,\ldots, x_k \right\},$ let $V'_y=V_y.$ We call $(T',\V')$ a\footnote{The splitting of a tree decomposition with respect to the same vertex $x$ is not unique, because for every neighbour $y_i,$ the vertex in $\{x_1,x_2,\ldots,x_k\}$ such that $y_i$ is connected to is not unique. Different connections might result in different tree decompositions. } splitting of $(T,\V)$ with respect to $x.$ 
\end{defi}
Figure~\ref{fig:split} illustrates the splitting operation with respect to a vertex of degree 4 {and} Figure~\ref{fig:split1} {illustrates} a splitting of the clique tree in Figure~\ref{fig:clique tree} with respect to $\{1,2,5,6\}$. The following {lemma} shows that the splitting operation results in another tree decomposition with the same width and fewer high-degree vertices.

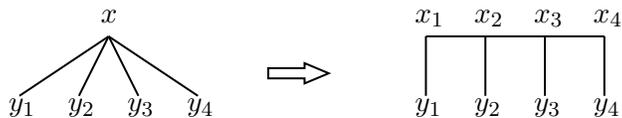
\begin{figure}
\centering
\begin{tikzpicture}[x=0.75pt,y=0.75pt,yscale=-1,xscale=1]

\draw    (140,190) -- (95,220) ;
\draw    (140,190) -- (125,220) ;
\draw    (140,190) -- (155,220) ;
\draw    (140,190) -- (185,220) ;
\draw    (300,190) -- (330,190) ;

\draw    (330,190) -- (360,190) ;

\draw    (360,190) -- (390,190) ;

\draw    (300,190) -- (300,220) ;

\draw    (330,190) -- (330,220) ;

\draw    (360,190) -- (360,220) ;

\draw    (390,190) -- (390,220) ;

\draw   (220.5,206.06) -- (238.8,206.06) -- (238.8,203.5) -- (251,208.63) -- (238.8,213.75) -- (238.8,211.19) -- (220.5,211.19) -- cycle ;
\draw (134.5,175) node [anchor=north west][inner sep=0.75pt]   [align=left] {$x$};
\draw (88,220) node [anchor=north west][inner sep=0.75pt]   [align=left] {$y_1$};
\draw (118,220) node [anchor=north west][inner sep=0.75pt]   [align=left] {$y_2$};
\draw (148,220) node [anchor=north west][inner sep=0.75pt]   [align=left] {$y_3$};
\draw (178,220) node [anchor=north west][inner sep=0.75pt]   [align=left] {$y_4$};
\draw (293,175) node [anchor=north west][inner sep=0.75pt]   [align=left] {$x_1$};
\draw (323,175) node [anchor=north west][inner sep=0.75pt]   [align=left] {$x_2$};
\draw (353,175) node [anchor=north west][inner sep=0.75pt]   [align=left] {$x_3$};
\draw (383,175) node [anchor=north west][inner sep=0.75pt]   [align=left] {$x_4$};
\draw (293,220) node [anchor=north west][inner sep=0.75pt]   [align=left] {$y_1$};
\draw (323,220) node [anchor=north west][inner sep=0.75pt]   [align=left] {$y_2$};
\draw (353,220) node [anchor=north west][inner sep=0.75pt]   [align=left] {$y_3$};
\draw (383,220) node [anchor=north west][inner sep=0.75pt]   [align=left] {$y_4$};
\end{tikzpicture}
\caption{Splitting vertex $x$ with degree 4 into a path with four vertices $x_1,x_2,x_3,x_4.$ Each of its neighbours is connected to one vertex on the path.}
\label{fig:split}
\end{figure}

\begin{lem}\label{splitprop}
Consider a graph $G=([n],\E)$ and a corresponding tree decomposition $(T,\V:=\{V_t\:\ t\in T\}).$ For any $x\in T$ with degree $k\geq 4,$ let $(T',\V')$ be a splitting of $(T,\V)$ with respect to $x.$ Then $(T',\V')$ is a tree decomposition of $G$ such that {${\rm wid}\(T',\V'\)={\rm wid}(T,\V).$} Moreover, we have that the number of vertices in $T'$ of degree greater than $3$ is that of $T$ minus 1.
\end{lem}

\begin{proof}
First, it is easy to verify that the vertex cover, edge cover and running intersection properties are all satisfied for $T'$ because all the newly added vertices in $T'$ have the same vertex set in $G$. Also, the splitting procedure doesn't change the width of the tree decomposition because the vertex sets in $\V'$ have the same sizes as before. After a splitting, the original vertex in $T$ of degree $>3$ is transferred to several vertices with degree $\leq 3.$ Also, the splitting won't change the degree of the other vertices. Thus, the number of vertices of degree $>3$ is decreased by $1.$
\end{proof}

{The following lemma gives an upper bound on the number of vertices after splitting all high degree vertices of a tree.}

\begin{lem}\label{lem:split}
For any $n\in \N^+,$ suppose $T=([n],\E)$ is a tree and $T'$ is obtained by splitting vertices of $T$ with degree greater than 3. Then $|T'|\leq 2n.$
\end{lem}
\begin{proof}
Suppose $T=T_0,T_1,T_2,\ldots,T_h=T'$ is the sequence of trees such that for any $i\in [h],$ $T_i$ is obtained by splitting some vertex of $T_{i-1}$ with degree greater than 3. For any $i\in \{0\}\sqcup [h],$ we define {the} edge sets $\E_1(T_i)$ and $\E_2(T_i)$ recursively as follows:

\begin{itemize}
\item $\E_1(T_0)=\E,$ $\E_2(T_0)=\emptyset.$
\item For $i=1,2,\ldots,h:$ suppose $T_i$ is obtained by splitting $x\in T_{i-1}$ into $x_1,x_2,\ldots,x_k$ such that $k\geq 4,$ $x_jx_{j+1}\in \E(T_i)$ and $x_jy_j\in \E(T_i)$ (see Figure~\ref{fig:split}). {Define}
\begin{align}\label{splitprocess}
\E_1(T_i)&:=\E_1(T_{i-1}\setminus\{x\})\sqcup\{ x_jy_j:\ j\in [k] \}\notag \\\E_2(T_i)&:=\E_2(T_{i-1}\setminus\{x\})\sqcup\{ x_jx_{j+1}:\ j\in [k-1] \}.
\end{align}
\end{itemize}
Intuitively speaking, $\E_2(T_i)$ is the set of edges in the path from the splitting of each vertex. $\E_1(T_i)$ is the set of the other edges. Because when we split a vertex, the new edges all come from the path, we have that $|\E_1(T_i)|=|\E_1(T_0)|=n-1.$ Also, for any $x\in T_i,$ it is the endpoint of at least one edge in $\E_1(T_i)$ as long as $n>1.$ Therefore, by counting the endpoints in $\E_1(T_h),$ we have that $|T'|=|T_h|\leq 1+2|\E_1(T_h)|=2n-1<2n.$ 
\end{proof}

With {Lemma~\ref{splitprop}} {and Lemma~\ref{lem:split}}, we are able to show that every graph has a binary tree decomposition.

\begin{coro}\label{splitcoro}
Consider a graph $G=([n],\E).$ {Suppose it has a (clique) tree decomposition $(T_0,\V_0)$ such that ${\rm wid}(T_0,\V_0)=\tw(G)$. Then it has a binary (clique) tree decomposition $(T,\V)$ such that ${\rm wid}(T,\V)=\tw(G)$ and $|T|\leq 2|T_0|.$}
\end{coro}

\begin{proof}
We apply the splitting operations to those vertices in the tree decomposition {$(T_0,\V_0)$} with degree $>3$ one by one until there are no vertices with degree greater $>3$ to get $(T,\V).$ From {Lemma~\ref{splitprop}}, $T$ is a binary tree decomposition with width equals to $\tw(G)$. From Lemma~\ref{lem:split}, we have that $|T|\leq {2|T_0|}.$ Similarly, if $G$ is a chordal graph, we can start from a clique tree and apply the splitting operations to obtain a binary clique tree decomposition with width $\tw(G)$.
\end{proof}

{We refer the reader to} Figure~\ref{fig:split1} and Figure~\ref{fig:split2} on how to split the clique tree in Figure~\ref{fig:clique tree} into a binary clique tree decomposition. {Readers might wonder why we need a {\em binary} tree decomposition. This is because our sparse extension of (\ref{SDP}) in Subsection~\ref{Subsec:spex} is an SDP problem whose sparsity pattern has tree-width proportional to the tree's maximum degree. We employ a splitting procedure to convert the tree decomposition into a binary tree decomposition to minimize the tree's maximum degree.}

\begin{figure}
\centering
\tikzset{every picture/.style={line width=0.75pt}} 

\begin{tikzpicture}[x=0.75pt,y=0.75pt,yscale=-1,xscale=1]

\draw    (183.5,378.25) -- (210,378.75) ;
\draw    (260,378.25) -- (282,378.25) ;
\draw    (148.5,349.25) -- (148.5,367.25) ;
\draw    (231.5,349.25) -- (231.5,368.75) ;
\draw    (302.5,385.75) -- (302.5,404.25) ;
\draw    (332.5,377.75) -- (355.5,377.75) ;
\draw    (374.5,350.25) -- (374.5,369.75) ;
\draw    (332,411.25) -- (355,411.25) ;
\draw    (262,411.75) -- (283,411.75) ;
\draw    (302.5,420.25) -- (302.5,438.75) ;

\draw (130.5,368) node [anchor=north west][inner sep=0.75pt]   [align=left] {\{1,2,5,6\}};
\draw (207,368.5) node [anchor=north west][inner sep=0.75pt]   [align=left] {\{1,2,5,6\}};
\draw (130.5,332) node [anchor=north west][inner sep=0.75pt]   [align=left] {\{1,2,8\}};
\draw (280,403.5) node [anchor=north west][inner sep=0.75pt]   [align=left] {\{2,3,4,5\}};
\draw (279,368.5) node [anchor=north west][inner sep=0.75pt]   [align=left] {\{1,2,5,6\}};
\draw (215.5,331.5) node [anchor=north west][inner sep=0.75pt]   [align=left] {\{1,6,7\}};
\draw (350.5,368) node [anchor=north west][inner sep=0.75pt]   [align=left] {\{1,2,5,6\}};
\draw (354,332.5) node [anchor=north west][inner sep=0.75pt]   [align=left] {\{5,6,12\}};
\draw (352,402) node [anchor=north west][inner sep=0.75pt]   [align=left] {\{2,3,9\}};
\draw (213.75,401.5) node [anchor=north west][inner sep=0.75pt]   [align=left] {\{4,5,11\}};
\draw (283,438.5) node [anchor=north west][inner sep=0.75pt]   [align=left] {\{3,4,10\}};

\end{tikzpicture}
\caption{A splitting of the clique tree in Figure~\ref{fig:clique tree} with respect to $\{1,2,5,6\}.$ What we obtain is a new clique tree decomposition with the same width and 1 fewer vertex of degree greater than $3.$}
\label{fig:split1}
\end{figure}
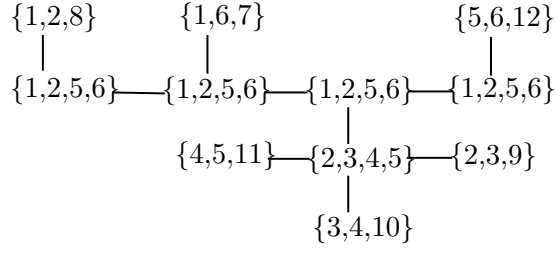

\begin{figure}
\centering
\tikzset{every picture/.style={line width=0.75pt}} 

\begin{tikzpicture}[x=0.75pt,y=0.75pt,yscale=-1,xscale=1]

\draw    (184.5,378.25) -- (211,378.75) ;
\draw    (260,378.25) -- (282,378.25) ;
\draw    (148.5,349.25) -- (148.5,367.25) ;
\draw    (231.5,349.25) -- (231.5,368.75) ;
\draw    (302.5,385.75) -- (302.5,404.25) ;
\draw    (332,378.25) -- (353,378.25) ;
\draw    (374.5,350.25) -- (374.5,369.75) ;
\draw    (332.5,411.75) -- (354,411.75) ;
\draw    (184.5,413.25) -- (211,413.75) ;
\draw    (260,413.25) -- (282,413.25) ;
\draw    (148.5,419.75) -- (148.5,437.75) ;
\draw    (232,420.25) -- (232,438.25) ;
\draw    (374,418.75) -- (374,436.75) ;

\draw (131,368.5) node [anchor=north west][inner sep=0.75pt]   [align=left] {\{1,2,5,6\}};
\draw (207,368.5) node [anchor=north west][inner sep=0.75pt]   [align=left] {\{1,2,5,6\}};
\draw (130.5,332) node [anchor=north west][inner sep=0.75pt]   [align=left] {\{1,2,8\}};
\draw (280,403.5) node [anchor=north west][inner sep=0.75pt]   [align=left] {\{2,3,4,5\}};
\draw (279,368.5) node [anchor=north west][inner sep=0.75pt]   [align=left] {\{1,2,5,6\}};
\draw (215.5,331.5) node [anchor=north west][inner sep=0.75pt]   [align=left] {\{1,6,7\}};
\draw (351,368.5) node [anchor=north west][inner sep=0.75pt]   [align=left] {\{1,2,5,6\}};
\draw (354,332.5) node [anchor=north west][inner sep=0.75pt]   [align=left] {\{5,6,12\}};
\draw (131.25,437) node [anchor=north west][inner sep=0.75pt]   [align=left] {\{4,5,11\}};
\draw (351.5,402) node [anchor=north west][inner sep=0.75pt]   [align=left] {\{2,3,4,5\}};
\draw (130.5,403) node [anchor=north west][inner sep=0.75pt]   [align=left] {\{2,3,4,5\}};
\draw (207,403.5) node [anchor=north west][inner sep=0.75pt]   [align=left] {\{2,3,4,5\}};
\draw (214.75,437.5) node [anchor=north west][inner sep=0.75pt]   [align=left] {\{3,4,10\}};
\draw (356.75,436) node [anchor=north west][inner sep=0.75pt]   [align=left] {\{2,3,9\}};

\end{tikzpicture}
\caption{A splitting of the clique tree decomposition in Figure~\ref{fig:split1} with respect to $\{2,3,4,5\}.$ What we obtain is a {\em binary} clique tree decomposition with width 3, which is the tree-width of the graph in Figure~\ref{fig:chordal graph}.}
\label{fig:split2}
\end{figure}
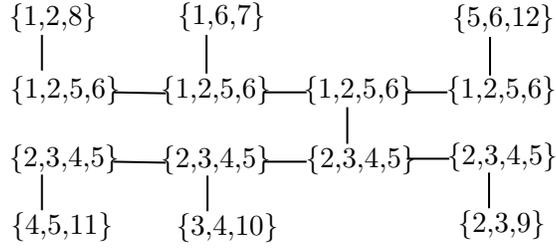

\section{Sparse extension of SDP with SPLR structure}\label{Sec:SET}

In this section, we {describe} our construction of sparse extension that decomposes {a} low-rank data matrix into several sparse data matrices. Because the sparse extension is {a} complicated {procedure} involving {many} matrix operations, we will first use a simple example to illustrate the main idea and later move on to the general case.


\subsection{A simple example}\label{Sec-simex}
Consider the following SDP problem:
\begin{equation}\label{simex}
\B:=\arg\min\left\{ 0:\ \dd(X)=e,\ \<{\bf a}{\bf a}^\top, X\> =b,\ X\in \S^n_+\right\},
\end{equation}
where ${\bf a}=[a_1;a_2;\ldots;a_n]\in \R^n$ is a nonzero vector, $b\in \R.$ It is easy to see that the SDP problem (\ref{simex}) has $\(G=\([n],\emptyset\),1\)-$SPLR structure with the sparse data matrices in the constraint $\dd(X)=e$ and a low-rank  data matrix ${\bf a}{\bf a}^\top.$ 

Now we are going to add auxiliary variables to transfer the low-rank constraint $\<{\bf a}{\bf a}^\top, X\> =b$ into sparse constraints. Suppose $X\in \B$ and $X=RR^\top$ for some $R\in \R^{n\times r},$ $r\in \N^+.$ Then we have that $\dd(RR^\top)=e$ and ${\bf a}^\top RR^\top {\bf a}=b.$ Consider $\widehat{R}\in \R^{2n\times r}$ such that $\widehat{R}_{1:n,:}:=R$ and the rest of the rows of $\widehat{R}$ are defined recursively as follows:
\begin{equation}\label{recurex}
{\rm For}\ i=1,2,\ldots,n:\ \widehat{R}_{n+i,:}:=\begin{cases} a_1 \widehat{R}_{1,:} & i=1 \\  \widehat{R}_{n+i-1,:}+a_i \widehat{R}_{i,:} & i\geq 2,\end{cases}
\end{equation}
which implies that $\widehat{R}_{2n,:}={\bf a}^\top \widehat{R}_{1:n,:}={\bf a}^\top R.$ Note that $\widehat{R}$ has the following matrix form
\begin{equation}\label{whatR}
\widehat{R}=\begin{bmatrix}R \\ a_1 R_{1,:} \\ a_1 R_{1,:}+a_2 R_{2,:}\\ \vdots \\ a_1 R_{1,:}+\cdots +a_n R_{n,:} \end{bmatrix}.
\end{equation}
Define $\widehat{X}=\widehat{R}\widehat{R}^\top.$ From (\ref{whatR}), we can see that the constraint $\<{\bf a}{\bf a}^\top,X\>=b$ is equivalent to $\widehat{X}_{2n, 2n}=b.$ Define the following new SDP problem:
\begin{multline}\label{simex1}
\widehat{\B}:=\arg\min\Big\{ 0:\ \dd\big(\widehat{X}_{1:n,1:n}\big)=e,\ \widehat{X}_{2n,2n}=b,\ [a_1,-1]\widehat{X}_{(1,n+1),(1,n+1)}[a_1;-1]=0,\\
 \forall i\in [n]\setminus \{1\},\ [a_i,1,-1]\widehat{X}_{(i,n+i-1,n+i),(i,n+i-1,n+i)}[a_i;1;-1]=0,\ \widehat{X}\in \S^{2n}_+\Big\},
\end{multline}
which has an optimal solution $\widehat{X}.$ {We recall that $[a_i,1,-1]$ and $[a_i;1;-1]$ denote a 3-dimensional row and column vector, respectively.}

\begin{claim}\label{simplecl}
The problem (\ref{simex1}) is an extension of (\ref{simex}) with $\I=[n].$
\end{claim}
\begin{proof}
For any $X\in \B$ such that $X=RR^\top,$ we know that $\widehat{X}=\widehat{R}\widehat{R}^\top\in \widehat{\B}$ where $\widehat{R}$ comes from (\ref{whatR}) and $\widehat{X}_{1:n,1:n}=X.$ For any $\widehat{X}=\widehat{R}\widehat{R}^\top\in \widehat{\B},$ from the constraints in (\ref{simex1}), we know that the rows of $\widehat{R}$ satisfy the recursive formula (\ref{recurex}). Therefore, we have that ${\bf a}^\top \widehat{X}_{1:n,1:n} {\bf a}=\widehat{X}_{2n,2n}=b$ and so $\widehat{X}_{1:n,1:n}\in \B.$ This completes the proof of Claim~\ref{simplecl}.
\end{proof}

Problem (\ref{simex1}) is a sparse SDP problem with the underlying graph $\widehat{G}=\big([2n],\widehat{\E}\big)$ being a chordal graph such that $\tw(\widehat{G})=\pw(\widehat{G})=2$ (see Figure~\ref{fig:simple chordal}), which is exactly {the upper bound} $\tw(G)+2\ell$ in Theorem~\ref{theo:spex} {since $\tw(G)=0$ and $\ell=1$}. {For the chordal graph $\widehat{G}$ in Figure~\ref{fig:simple chordal}, we have that $\(T:=[n],\V\),$ where 
\begin{multline}\label{simtrdecmp}
\V=\Big\{ V_1=\{1,n+1\},\ V_2=\{2,n+1,n+2\},\ldots,V_{n-1}=\{n-1,2n-2,2n-1\},\\
 V_n=\{n,2n-1,2n\} \Big\}
\end{multline}
forms a binary clique path decomposition.} In addition, each block $\widehat{Y}_{(i,n+i-1,n+i),(i,n+i-1,n+i)}$ ($i\geq 2$) has rank at most $2$ because of the constraint 
\begin{equation}\label{consY33}
[a_i,1,-1]\widehat{Y}_{(i,n+i-1,n+i),(i,n+i-1,n+i)}[a_i;1;-1]=0.
\end{equation}
Therefore, we can apply algorithm 3.1 in \cite{jiang2017minimum} to find a solution $\widehat{Y}$ such that $\rr\big(\widehat{Y}\big)\leq 2.$ Thus, $Y:=\widehat{Y}_{1:n,1:n}$ is an optimal solution of (\ref{simex}) such that $\rr(Y)\leq 2,$ which is exactly the rank bound $\tw(G)+\ell+1$ in Theorem~\ref{ubthmtree}.

{We should mention that while the sparse extension in \eqref{simex1} has increased the matrix dimension from $n$ to $2n$, the chordal conversion of \eqref{simex1} will give an SDP containing $n$  positive semidefinite blocks of dimension at most $3$. Thus in fact there is a reduction in the memory cost from $n^2$ to $9n$.
}

\begin{figure}       
\centering
\tikzset{every picture/.style={line width=0.75pt}} 

\begin{tikzpicture}[x=0.75pt,y=0.75pt,yscale=-1,xscale=1]
\draw    (4,10) -- (4,38) ;
\draw    (44,10) -- (44,38) ;
\draw    (84,10) -- (84,38) ;
\draw    (124,10) -- (124,38) ;
\draw    (164,10) -- (164,38) ;
\draw    (204,10) -- (204,38) ;
\draw    (244,10) -- (244,38) ;
\draw    (284,10) -- (284,38) ;

\draw    (44,38) -- (4,38) ;
\draw    (84,38) -- (44,38) ;
\draw    (124,38) -- (84,38) ;
\draw    (204,38) -- (164,38) ;
\draw    (244,38) -- (204,38) ;
\draw    (284,38) -- (244,38) ;

\draw    (44,10)-- (4,38) ;
\draw    (84,10) -- (44,38) ;
\draw    (124,10) -- (84,38) ;

\draw (134,18) node [anchor=north west][inner sep=0.75pt]   [align=left] {{$\cdots$}};

\draw    (204,10) -- (164,38) ;
\draw    (244,10) -- (204,38) ;
\draw    (284,10) -- (244,38) ;

\draw (0,0) node [anchor=north west][inner sep=0.75pt]   [align=left] {{\tiny 1}};
\draw (40,0) node [anchor=north west][inner sep=0.75pt]   [align=left] {{\tiny 2}};
\draw (80,0) node [anchor=north west][inner sep=0.75pt]   [align=left] {{\tiny 3}};
\draw (120,0) node [anchor=north west][inner sep=0.75pt]   [align=left] {{\tiny 4}};
\draw (0,40) node [anchor=north west][inner sep=0.75pt]   [align=left] {{\tiny $n+1$}};
\draw (40,40) node [anchor=north west][inner sep=0.75pt]   [align=left] {{\tiny $n+2$}};
\draw (80,40) node [anchor=north west][inner sep=0.75pt]   [align=left] {{\tiny $n+3$}};
\draw (120,40) node [anchor=north west][inner sep=0.75pt]   [align=left] {{\tiny $n+4$}};
\draw (160,0) node [anchor=north west][inner sep=0.75pt]   [align=left] {{\tiny $n-3$}};
\draw (200,0) node [anchor=north west][inner sep=0.75pt]   [align=left] {{\tiny $n-2$}};
\draw (240,0) node [anchor=north west][inner sep=0.75pt]   [align=left] {{\tiny $n-1$}};
\draw (280,0) node [anchor=north west][inner sep=0.75pt]   [align=left] {{\tiny $n$}};
\draw (160,40) node [anchor=north west][inner sep=0.75pt]   [align=left] {{\tiny $2n-3$}};
\draw (200,40) node [anchor=north west][inner sep=0.75pt]   [align=left] {{\tiny $2n-2$}};
\draw (240,40) node [anchor=north west][inner sep=0.75pt]   [align=left] {{\tiny $2n-1$}};
\draw (280,40) node [anchor=north west][inner sep=0.75pt]   [align=left] {{\tiny $2n$}};
\end{tikzpicture}
\caption{Chordal graph $\widehat{G}=\big([2n],\widehat{\E}\big)$ with $\tw\big(\widehat{G}\big)=\pw\big(\widehat{G}\big)=2.$}
\label{fig:simple chordal}
\end{figure}
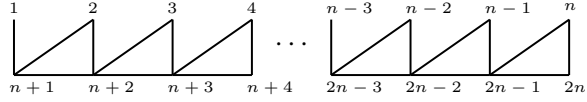

Although the problem (\ref{simex}) is a rather simple SDP problem, our later sparse extension for the general problem (\ref{SDP}) and the proof of Theorem~\ref{theo:spex} are {based on similar ideas (even though the technical details are much more complicated)} in the sense that we will introduce auxiliary variables to transfer low-rank constraints into sparse constraints. In the next subsection, we will provide the details.


\subsection{Construction of sparse extension for (\ref{SDP})}\label{Subsec:spex}

We now provide an explicit construction of sparse extension for (\ref{SDP}) that is similar to the simple case in Subsection~\ref{Sec-simex}. Consider ${\bf A}:=\left[ {\bf a}^1,{\bf a}^2,\ldots,{\bf a}^\ell \right]\in \R^{n\times \ell}$ such that the {column vectors form} the range {space} of $[A_0^l,A_1^l,A_2^l,\ldots,A_m^l].$ {Then we} have that for any $i\in \setm,$ there exists $S_i\in \S^{\ell}$ such that $A_i^l={\bf A}S_i{\bf A}^\top.$ 

From Remark~\ref{treerem}, we may assume that $G$ is a chordal graph. {Let $(T_0,\V_0)$ be a clique tree decomposition of $G$ such that ${\rm wid}(T_0,\V_0)=\tw(G)$ and $|T_0|=p.$ From Corollary~\ref{splitcoro}, $G$ has a binary clique tree decomposition $\(T,\V:=\{V_t:\ t\in T\}\)$ such that $ {\rm wid}(T,\V)=\tw(G)$ and $|T|\leq 2{p}.$ If $T_0$ is a path and $\pw(G)=\tw(G),$ then we directly let $(T,\V)=(T_0,\V_0)$ and we have that $|T|\leq {p}.$} For convenience, we use $[k]$ to denote the vertex set of $T.$ We choose a vertex {\bf with degree less than 3} in $T$ as its root to make it a rooted binary tree (see Figure~\ref{fig:tree decomposition lem} {in Appendix~\ref{app-illus}}). For any $x\in T,$ define ${\rm chd}(x)\subset T$ be the set of children of $x,$ ${\rm par}(x)\subset T$ be the set of parent of $x.$ We have that $|{\rm chd}(x)|\leq 2$ and $|{\rm par}(x)|\leq 1.$ When $T$ is a path, then we choose the root to be one of the end points of $T$ so that each node has at most 1 child. The construction contains the following four steps.

\medskip 

\noindent{\bf Step 1.} Partitioning the vertex set $[n]$ in $G$ according to the tree decomposition $(T,\V).$

\medskip

{The goal of this step is to partition every long vector ${\bf a}^h$ into several short vectors such that each one is contained in a block {(corresponding to $V_t$ for some $t\in T$)} in the tree-decomposition. We do this partition so that later we can transfer a low-rank constraint matrix into several sparse matrices.} 

For any $t\in [k],$ define the vertex set $W_t:=V_t\setminus\( \cup_{i\in {\rm par}(t)}V_i \),$ which is the set of vertices in $V_t$ that don't appear in its parent's block. From the running intersection property of tree decomposition, we have that $\{W_t:\ t\in [k]\}$ is a vertex partition of $G,$ i.e., $[n]=\sqcup_{t\in T}W_t$ (see Figure~\ref{fig:tree decomposition partition} {in Appendix~\ref{app-illus}}). 

With the partition of $[n],$ we can partition each vector ${\bf a}^h$ {of ${\bf A}$} into smaller vectors. For any $h\in [\ell],t\in [k],$ let ${\bf a}^{h,t}:={\bf a}^h_{W_t}\in \R^{W_t}$ be the subvector of ${\bf a}^h$ in the index set $W_t.$

\medskip

\noindent{\bf Step 2.} Extending $G$ to a larger graph with the same tree structure.

\medskip

{We extend $G$ to a larger graph because later we are going to add auxiliary variables to extend (\ref{SDP}) into a higher dimensional space. The new vertices willcorrespond to new matrix dimensions in the extended problem. The new edges will correspond to new constraint matrices in the extended problem.}

For any $t\in [k],$ define the vertex set $U_t:=\{n+(t-1)\ell+1,n+(t-1)\ell+2,\ldots,n+t\ell\}.$ Define a new graph $\widetilde{G}=\big( [n+k\ell],\widetilde{\E} \big)$ 
(see Figure~\ref{fig:tree extension} {in Appendix~\ref{app-illus}}) such that 
\begin{equation}\label{edgedefi}
\widetilde{\E}:=\E\sqcup\(\sqcup_{t\in [k]}E(U_t)\)\sqcup \(\sqcup_{t\in [k]}\cup_{i\in {\rm par}(t)\sqcup\{t\}} E(U_t,V_i)\),
\end{equation}
{where $E(U_t)$ denotes the set of $|U_t|(|U_t|-1)/2$ edges in $U_t$ and $E(U_t,V_i)$ denotes the $|U_t||V_i|$ edges between $U_t$ and $V_i.$} Here we don't require $\widetilde{G}$ to be a chordal graph. With the extended graph $\widetilde{G},$ we also extend the binary clique tree decomposition $(T,\V)$ of $G$ to a new binary tree decomposition $\big(T,\widetilde{\V}:=\big\{\widetilde{V}_t:\ t\in [k]\big\}\big)$ of $\widetilde{G}$ (see Figure~\ref{fig:tree decomposition extention} {in Appendix~\ref{app-illus}}) such that for any $t\in [k],$ 
\begin{equation}\label{blockeqv}
\widetilde{V}_t:=V_t\sqcup U_t \sqcup_{i\in {\rm chd}(t)}U_i.
\end{equation}
Note that $\big(T,\widetilde{\V}\big)$ may {\bf not} be a {\bf clique} tree decomposition. Indeed, from (\ref{edgedefi}) and (\ref{blockeqv}), we know that for any $t\in [k],$ the edge set in the subgraph $\widetilde{G}\big[ \widetilde{V}_t \big]$ is
\begin{equation}\label{edgesetblkt}
\E\big( \widetilde{G}\big[ \widetilde{V}_t \big] \big)=E(V_t)\sqcup \(E(U_t)\sqcup E(U_t,V_t)\)\sqcup \(\sqcup_{j\in {\rm chd}(t)}\( E(U_j)\sqcup E(U_j,V_t)\)\).
\end{equation}

\begin{claim}\label{claimtr}
$\big(T,\widetilde{\V}\big)$ is a binary tree decomposition of $\widetilde{G}$ with width at most $\tw(G)+3\ell$. Moreover, if $\pw(G)=\tw(G),$ and $(T,\V)$ is a path-decomposition, then $\tw(\widetilde{G})\leq \tw(G)+2\ell.$
\end{claim}

From (\ref{edgedefi}) and (\ref{blockeqv}), it is easy to see that the vertex cover, edge cover and running intersection properties in Definition~\ref{treedefi} are all satisfied. Thus, $\big(T,\widetilde{\V}\big)$ is a binary tree decomposition of $\widetilde{G}.$ For any $t\in [k],$ because $|U_t|=\ell,$ from (\ref{blockeqv}) and the fact that $|{\rm chd}(t)|\leq 2$, we get $|\widetilde{V}_t|\leq \tw(G)+3\ell+1.$ This implies that the ${\rm wid}\big(T,\widetilde{\V}\big)\leq \tw(G)+3\ell.$ Moreover, if $\pw(G)=\tw(G),$ and $(T,\V)$ is a path-decomposition, then every node in $T$ has at most one child, which implies that ${\rm wid}\big(T,\widetilde{\V}\big)\leq \tw(G)+2\ell.$ This completes the proof of Claim~\ref{claimtr}. 

\medskip

\noindent{\bf Step 3.} Extending matrix $X$ into higher dimensional space.

\medskip

{With the extended graph in Step 2, now we extend the variable $X$ into a higher dimensional space through a recursive procedure that is similar to (\ref{recurex}) for the simple example in Section \ref{Sec-simex}.}

Consider {a} feasible solution $X$ of (\ref{SDP}). Let $r:=\rr(X),$ because $X$ is positive semidefinite, it has the decomposition $X=RR^\top$ such that $R\in \R^{n\times r}.$ Now we extend $R$ to $\widetilde{R}\in \R^{(n+k\ell)\times r}$ recursively as follows. Consider an ordering $t_1,t_2,\ldots,t_{k}$ of the vertex set $[k]$ of $T$ such that parents always appear {\bf later} than their children. Such an ordering can be found by depth-first search. We recursively define the rows of $\widetilde{R}$ as follows:

\begin{itemize}
\item [1,] $\widetilde{R}_{1:n,:}:=R.$
\item [2,] For $t=t_1,t_2,\ldots,t_k:$
\begin{equation}\label{recurR}
\forall h\in [\ell],\ \quad \widetilde{R}_{n+(t-1)\ell+h,:}:={({\bf a}^{h,t})^\top} R_{W_t,:}+\sum_{j\in {\rm chd}(t)}\widetilde{R}_{n+(j-1)\ell+h,:}
\end{equation}
\end{itemize}
{Note that the} above recursive definition is similar to (\ref{recurex}). In (\ref{recurR}), when we define $\widetilde{R}_{n+(t-1)\ell+h,:},$ for $j\in {\rm chd}(t),$ $\widetilde{R}_{n+(j-1)\ell+h,:}$ has already been defined because parents appear later than their children. Thus, (\ref{recurR}) is well-defined. From (\ref{recurR}) and the tree structure of $T,$ we have that for any $h\in [\ell],$ $t\in [k],$
\begin{equation}\label{inductiontree}
\widetilde{R}_{n+(t-1)\ell+h,:} =\sum_{j\in T_t}{({\bf a}^{h,j})^\top} R_{W_j,:}\,,
\end{equation}
where $T_t$ is the subtree of $T$ with $t$ as its root. From (\ref{inductiontree}) and the partition property $[n]=\sqcup_{t\in [k]}W_t$ in Step 1, we have that
\begin{equation}\label{passequi}
\forall h\in [\ell],\ \widetilde{R}_{n+(k-1)\ell+h,:} ={({\bf a}^{h})^\top} R,
\end{equation}
which implies that 
\begin{equation}\label{Rsum}
\forall h_1,h_2\in [\ell],\ \widetilde{R}_{n+(k-1)\ell+h_1,:}\widetilde{R}_{n+(k-1)\ell+h_2,:}^\top={({\bf a}^{h_1})^\top}X {\bf a}^{h_2}.
\end{equation}
Define $\widetilde{X}:=\widetilde{R}\widetilde{R}^\top\in \S^{n+k\ell}_+$ as the extension of $X$ in the higher dimensional space $\S^{n+k\ell}_+.$ {Then from} (\ref{Rsum}), we have that
\begin{equation}\label{Xsum}
\forall h_1,h_2\in [\ell],\ \widetilde{X}_{n+(k-1)\ell+h_1,n+(k-1)\ell+h_2}={({\bf a}^{h_1})^\top} X {\bf a}^{h_2}.
\end{equation}
Moreover, because $\widetilde{R}_{1:n,:}=R,$ we have that
\begin{equation}\label{fX}
\widetilde{X}_{1:n,1:n}=X.
\end{equation}


\noindent{\bf Step 4.} {Extending (\ref{SDP}) into a higher dimensional space and transferring low-rank data matrices into sparse data matrices.}

\medskip

{In the last step, we are going to extend (\ref{SDP}) into a higher dimensional space such that the extended variable $\widetilde{X}$ is a feasible solution. The new problem is similar to (\ref{simex1}) because we decompose each low-rank constraint matrix into several sparse constraint matrices.}

For any $h\in [\ell],$ $t\in [k],$ define the vector $\tilde{\bf a}^{h,t}\in \R^{\widetilde{V}_t}$ such that for any $i\in \widetilde{V}_t$,
\begin{equation}\label{hadefi}
\tilde{\bf a}^{h,t}_i:=\begin{cases}
{\bf a}^{h,t}_i & i\in W_t\\
1 & i=n+(j-1)\ell+h,\ j\in {\rm chd}(t)\\
-1 & i=n+(t-1)\ell+h \\
0 & {\rm otherwise.}
\end{cases}
\end{equation}
Comparing (\ref{hadefi}) and (\ref{recurR}), we can see that (\ref{recurR}) is equivalent to that
\begin{equation}\label{haequ}
\forall h\in [\ell],\ t\in [k],\  {(\tilde{\bf a}^{h,t})^\top} \widetilde{R}_{\widetilde{V}_t,:}=0,
\end{equation}
which is further equivalent to that 
\begin{equation}\label{haequX}
\forall h\in [\ell],\ t\in [k],\  {(\tilde{\bf a}^{h,t})^\top} \widetilde{X}_{\widetilde{V}_t,\widetilde{V}_t}\tilde{\bf a}^{h,t}=0.
\end{equation}
For notational convenience, let $\I$ be the index set $[n]$ and $\J$ be the index set {$\{n+(k-1)\ell+1,\ldots,n+k\ell\}.$} From (\ref{Xsum}) and (\ref{fX}), we have that 
\begin{equation}\label{lscons}
\forall i\in \setm,\ \<A_i,X\>=\<A_i^s,X\>+\<A_i^l,X\>=\<A_i^s,\widetilde{X}_{\I,\I}\>+\<S_i,\widetilde{X}_{\J,\J}\>.
\end{equation}
Similar to (\ref{simex1}), we define the following new problem,
\begin{multline}\label{newprob}
\min\Big\{ \<A_0^s,\widetilde{X}_{\I,\I}\>+\<S_0,\widetilde{X}_{\J,\J}\>:\ \forall i\in [m]\ b_i^l\leq \<A_i^s,\widetilde{X}_{\I,\I}\>+\<S_i,\widetilde{X}_{\J,\J}\>\leq b_i^u, 
\\
 \forall h\in [\ell],\ t\in [k],\  {(\tilde{\bf a}^{h,t})^\top} \widetilde{X}_{\widetilde{V}_t,\widetilde{V}_t}\tilde{\bf a}^{h,t}=0,\ \widetilde{X}\in \S^{n+k\ell}_+\Big\}.
\end{multline}
Similar to Claim~\ref{simplecl}, we have the following claim.


\begin{claim}\label{claimex}
The problem (\ref{newprob}) is an extension of (\ref{SDP}). 
\end{claim}
\begin{proof}
For any feasible solution $X=RR^\top$ of (\ref{SDP}), we can define $\widetilde{X}{=\tilde{R}\tilde{R}^\top}$ following the process of (\ref{recurR}). {Then from} (\ref{fX}), $\widetilde{X}_{\I,\I}=X.$ {Now from} (\ref{lscons}), $\widetilde{X}$ is {a} feasible solution of (\ref{newprob}) and $\<A_0,X\>=\<A_0^s,\widetilde{X}_{\I,\I}\>+\<S_0,\widetilde{X}_{\J,\J}\>.$ {Conversely, for} any feasible $\widetilde{X}=\widetilde{R}\widetilde{R}^\top$ of (\ref{newprob}), because of the constraints $(\tilde{\bf a}^{h,t})^\top \widetilde{X}_{\widetilde{V}_t,\widetilde{V}_t}\tilde{\bf a}^{h,t}=0$ in (\ref{newprob}) and the equivalence between (\ref{recurR}), (\ref{haequ}) and (\ref{haequX}), we have that the recursive formula (\ref{recurR}) is still satisfied for the row vectors of $\widetilde{R}.$ Therefore, we have that (\ref{lscons}) hold{s}. {This} implies that {$\widetilde{X}_{\I,\I}$} is feasible for (\ref{SDP}) with the same function value as that of $\widetilde{X}$ in (\ref{newprob}). This completes the proof of Claim~\ref{claimex}. 
\end{proof}

\subsection{Proof of Theorem~\ref{theo:spex}}
\begin{proof}

With the sparse extension (\ref{newprob}) and Claim~\ref{claimex} in Subsection~\ref{Subsec:spex}, in order to prove Theorem~\ref{theo:spex}, we only have to show that the tree-width and matrix dimension of the problem (\ref{newprob}) satisfy the condition{s} in Theorem~\ref{theo:spex}. From the data matrices in the problem (\ref{newprob}) and (\ref{hadefi}), we can see that the sparsity pattern of (\ref{newprob}) is contained in the chordal completion of $\widetilde{G}$ (see Figure~\ref{fig:tree decomposition completion} {in Appendix~\ref{app-illus}}). From (\ref{blockeqv}) and the fact that $T$ is a binary tree, we have that $\tw(\widetilde{G})\leq \tw(G)+3\ell$ and the matrix dimension of (\ref{newprob}) {is at most $n+2p\cdot\ell$ because $|T|\leq 2p$.} When $\pw(G)=\tw(G)$ {and $T$ is a path,} because every node has at most $1$ child, we have that $\tw(\widetilde{G})\leq \tw(G)+2\ell$ and Also, the matrix dimension of (\ref{newprob}) {is at most $n+p\cdot\ell$ because $|T|\leq p$}. This completes the proof of Theorem~\ref{theo:spex}.
\end{proof}

\subsection{Chordal conversion of sparse extension}

In Subsection~\ref{Subsec:spex}, we discussed how to convert (\ref{SDP}) into a sparse SDP problem with bounded tree-width (\ref{newprob}). In practice, we can apply chordal conversion to transfer it into a multi-block SDP problem. In detail, let 
$\widehat{G}:=\big(n+k\ell,\widehat{\E}\big)$ be the chordal completion of $\widetilde{G}$ (see Figure~\ref{fig:tree decomposition completion} {in Appendix~\ref{app-illus}}). {According} to  Grone’s and Dancis's theorems\cite{grone1984positive,dancis1992positive}, problem (\ref{newprob}) is equivalent to the following problem{:}
\begin{multline}\label{newprob1}
\min\Big\{ \<A_0^s,\widetilde{X}_{\I,\I}\>+\<S_0,\widetilde{X}_{\J,\J}\>:\ \forall i\in [m],\ b_i^l\leq \<A_i^s,\widetilde{X}_{\I,\I}\>+\<S_i,\widetilde{X}_{\J,\J}\>\leq b_i^u, 
\\
 \forall h\in [\ell],\ t\in [k],\  {(\tilde{\bf a}^{h,t})^\top} \widetilde{X}_{\widetilde{V}_t,\widetilde{V}_t}\tilde{\bf a}^{h,t}=0,\ \widetilde{X}_{\widetilde{V}_t,\widetilde{V}_t}\in \S^{|\widetilde{V}_t|}_+, \;\widetilde{X}\in \S(\widehat{\E},0)\Big\},
\end{multline}
where the PSD constraint of the whole matrix variable is decomposed into the PSD constraints of the corresponding smaller blocks in the tree decomposition $(T,\widetilde{\V}).$ Moreover, $\widetilde{X}$ is a sparse matrix with 
sparsity pattern contained in $\widehat{E}.$

To solve (\ref{newprob1}), we can use the alternating {direction} method of multiplier (ADMM) by {considering} the following equivalent formulation:  
\begin{multline}\label{newprob2}
\min\Big\{ \<A_0^s,\widetilde{X}_{\I,\I}\>+\<S_0,\widetilde{X}_{\J,\J}\>:\ \forall i\in [m],\ b_i^l\leq \<A_i^s,\widetilde{X}_{\I,\I}\>+\<S_i,\widetilde{X}_{\J,\J}\>\leq b_i^u, 
\\
 \forall h\in [\ell],\ t\in [k],\ \widetilde{X}_{\widetilde{V}_t,\widetilde{V}_t}=Y_t,\ {(\tilde{\bf a}^{h,t})^\top} Y_t\tilde{\bf a}^{h,t}=0,\ Y_t\in \S^{|\widetilde{V}_t|}_+, \; \widetilde{X}\in \S(\widehat{\E},0)\Big\},
\end{multline}
where we introduce {the} variables $\{Y_t\}_{t\in T}$ to separate the linear constraints and PSD constraints. Another popular method for {solving} sparse SDP is interior point method (IPM). However, IPM is not applicable to solve (\ref{newprob2}) directly because of the constraints ${(\tilde{\bf a}^{h,t})^\top} Y_t\tilde{\bf a}^{h,t}=0$, which implies that there is no {strict primal} feasible solution. In order to alleviate this issue, we can apply facial reduction \cite{wolkowicz2012handbook} to remove the constraint ${(\tilde{\bf a}^{h,t})^\top} Y_t\tilde{\bf a}^{h,t}=0.$

\section{The rank bound of SDP with SPLR structure}\label{Sec:rank}


\subsection{Proof of Theorem~\ref{ubthmtree}.}\label{Subsec:pfthm}

\begin{proof}

Consider the problem (\ref{newprob}), which is the sparse extension of (\ref{SDP}).

When $\pw(G)=\tw(G),$ we use the path decomposition $(T,\V)$ of $G$ and $\widetilde{G}$ defined in (\ref{edgedefi}). Consider the chordal extension 
$\widehat{G}=\big({[n+k\ell]},\widehat{\E}\big)$ of $\widetilde{G}.$ We have that $\tw(\widehat{G})=\pw(\widehat{G})=\tw(G)+2\ell$. Moreover, $\widehat{G}$ is the sparsity pattern of (\ref{newprob}). From (\ref{blockeqv}), we know that for any $t\in [k],$ the block size $|\widetilde{V}_t|\leq \tw(G)+2\ell+1.$ Let $\widetilde{X}$ be an {\em optimal} solution of (\ref{newprob}) and consider the following set 

\begin{equation}\label{Sset}
\S_+\big(\widetilde{X},\widehat{G}\big):=\left\{ \widetilde{Y}\in \S^{n+k\ell}_+:\ \P_{\widehat{\E}}\big(\widetilde{Y}\big)
=\P_{\widehat{\E}}\big(\widetilde{X}\big)\right\}.
\end{equation}
Because $\widehat{G}$ is the sparsity pattern of (\ref{newprob}), we know that every element of $\S_+\big(\widetilde{X},\widehat{G}\big)$ is an optimal solution of (\ref{newprob}). Moreover, from (\ref{haequX}) and the fact that for any $t\in [k],$ the vectors in $\left\{\hat{\bf a}^{h,t}:\ h\in [\ell]\right\}$ are linearly independent, we have that for any $t\in [k],$ $\rr\big(\widehat{X}_{\widehat{V}_t,\widehat{V}_t}\big)\leq \tw(G)+\ell+1$. Therefore, we can apply algorithm 3.1 in \cite{jiang2017minimum} to find an element $\widetilde{Y}\in\S_+\big(\widetilde{X},\widehat{G}\big)$ whose rank is at most $\pw(G)+\ell+1.$ {Thus} $\widetilde{Y}_{1:n,1:n}$ is an optimal solution of (\ref{SDP}) with rank at most $\pw(G)+\ell+1.$

Now we move on to the more general case where $\tw(G)$ is not necessarily equal to $\pw(G).$ In this case, because $T$ is a binary tree, a node might have two children, we have that $|\widetilde{V}_t|\leq \tw(G)+3\ell+1$ from (\ref{blockeqv}). Similarly, we can consider the chordal completion of $\widetilde{G}$ to get the sparsity pattern $\widehat{G}$ of (\ref{newprob}) (see Figure~\ref{fig:tree decomposition completion} in Appendix A) and apply algorithm 3.1 in \cite{jiang2017minimum} to find a low-rank element $\S_+\big(\widetilde{X},\widehat{G}\big)$. However, by doing this, we will get the rank bound of
$\tw(G)+2\ell+1$ instead of $\tw(G)+\varphi(\ell)+1$. In order to get the tighter rank bound {of} $\tw(G)+\varphi(\ell)+1,$ we will first apply a rank reduction procedure in each 
smaller block $\widetilde{X}_{\widetilde{V}_t,\widetilde{V}_t}$ independently and then apply chordal completion to $\widetilde{G}.$ Here is where the function $\varphi(\cdot)$ in Definition~\ref{cvxnb} comes into play. Consider the following set
\begin{multline}\label{tildeSset}
\widetilde{\S}_+\big(\widetilde{X},\widetilde{G}\big):=\Big\{ \widetilde{Y}\in \S^{n+k\ell}_+:\ \P_{\widetilde{\E}}\big(\widetilde{Y}\big)=\P_{\widetilde{\E}}\big(\widetilde{X}\big),
\\
 \forall h\in [\ell],\ t\in [k],\  {(\tilde{\bf a}^{h,t})^\top} \widetilde{Y}_{\widetilde{V}_t,\widetilde{V}_t}\tilde{\bf a}^{h,t}=0\Big\},
\end{multline}
where $\tilde{\bf a}^{h,t}$ is defined as in \eqref{hadefi}.
It is easy to see that any element in (\ref{tildeSset}) is an optimal solution of (\ref{newprob}). For any $t\in [k],$ define the matrix 
\begin{equation}\label{matbfA}
{\bf A}^t:=\left[ \tilde{\bf a}^{1,t},\tilde{\bf a}^{2,t},\ldots, \tilde{\bf a}^{\ell,t}\right]\in \R^{\widetilde{V}_t\times\ell}.
\end{equation}
We have the following equivalence,
\begin{equation}\label{condYbfA}
\forall h\in [\ell],\ t\in [k],\  {(\tilde{\bf a}^{h,t})^\top} \widetilde{Y}_{\widetilde{V}_t,\widetilde{V}_t}\tilde{\bf a}^{h,t}=0 \;\;
\Leftrightarrow \;\;
\forall t\in [k],\ {\bf A}^{t\top} \widetilde{Y}_{\widetilde{V}_t,\widetilde{V}_t}{\bf A}^{t}={\bf 0}_{\ell\times \ell}.
\end{equation}
Moreover, from (\ref{hadefi}), we have that 
\begin{equation}\label{matA}
\forall t\in [k],\ \begin{cases}
{\bf A}^t_{W_t,:}={\bf A}_{W_t,:}, &\\
{\bf A}^t_{U_j,:}=I_\ell, & \forall j\in {\rm chd}(t)\\
{\bf A}^t_{U_t,:}=-I_\ell,&\\
{\bf A}^t_{V_t\setminus W_t,:}={\bf 0}_{(|V_t|-|W_t|)\times \ell}. &
\end{cases}
\end{equation}
For any $t\in [k],$ define the following set
\begin{multline}\label{setbfA}
\widetilde{\S}_+\big(\widetilde{X},t\big):=\Big\{ Z^t\in \S^{\widetilde{V}_t}_+:\ \P_{\E\(\widetilde{G}\left[\widetilde{V_t}\right]\)}\( Z^t \)=\P_{\E\big(\widetilde{G}\left[\widetilde{V_t}\right]\big)}\big( \widetilde{X}_{\widetilde{V}_t,\widetilde{V}_t} \big),
\;\;
 {\bf A}^{t\top} Z^t {\bf A}^t={\bf 0}_{\ell\times \ell}
 \Big\}.
\end{multline}
From (\ref{tildeSset}) and (\ref{condYbfA}), we have that 
\begin{equation}\label{equivS}
\widetilde{\S}_+\big(\widetilde{X},\widetilde{G}\big)= \left\{ \widetilde{Y}\in \S^{n+k\ell}_+:\ \forall t\in [k],\ \widetilde{Y}_{\widetilde{V}_t,\widetilde{V}_t}\in \widetilde{\S}_+\big(\widetilde{X},t\big) \right\}.
\end{equation}
The following claim is the rank reduction procedure mentioned before. 

\begin{claim}\label{claimStilde}
For any $t\in [k],$ $\widetilde{\S}_+(\widetilde{X},t)$ contains an element of rank at most $\tw(G)+\varphi(\ell)+1.$
\end{claim}
{First,} $\widetilde{\S}_+(\widetilde{X},t)\neq \emptyset$ because $\widetilde{X}_{\widetilde{V}_t,\widetilde{V}_t}\in \widetilde{\S}_+(\widetilde{X},t).$ If $|{\rm chd}(t)|\leq 1,$ then from (\ref{blockeqv}), $|\widetilde{V}_t|\leq \tw(G)+2\ell+1.$ Because of the constraints $ {\bf A}^{t\top} Z^t {\bf A}^t={\bf 0}_{\ell\times \ell}$ and the fact that $\rr\({\bf A}^{t}\)=\ell$ (see (\ref{matA})), we have that any element $Z^t\in \widetilde{\S}_+(\widetilde{X},t)$ has rank at most $\tw(G)+\ell+1\leq \tw(G)+\varphi(\ell)+1.$ If ${\rm chd}(t)=\{j_1,j_2\},$ by rearranging the rows and columns of $Z^t\in \S^{\widetilde{V}_t}_+$ in the order $V_t,U_{j_1},U_{j_2},U_t,$ the linear constraints in $\widetilde{\S}_+(\widetilde{X},t)$ can be written as follows:
\begin{equation}\label{matSt}
[{\bf A}^{t\top}_{V_t},I_\ell,I_\ell,-I_\ell]\(Z^t=\begin{bmatrix}
\widetilde{X}_{V_t,V_t} & \widetilde{X}_{V_t,U_{j_1}} & \widetilde{X}_{V_t,U_{j_2}} & \widetilde{X}_{V_t,U_t}\\
\widetilde{X}_{U_{j_1},V_t} & \widetilde{X}_{U_{j_1},U_{j_1}} & * & *\\
\widetilde{X}_{U_{j_2},V_t} & * & \widetilde{X}_{U_{j_2},U_{j_2}} & * \\
\widetilde{X}_{U_t,V_t} & * & * & \widetilde{X}_{U_t,U_t}
\end{bmatrix}\)\begin{bmatrix}{\bf A}^{t}_{V_t}\\I_\ell\\I_\ell\\-I_\ell\end{bmatrix}={\bf 0}_{\ell\times \ell},
\end{equation}
where we have used (\ref{matA}) in deriving the above matrix formulation. Using Lemma~\ref{rankreduc} (note that (\ref{matSt}) and (\ref{3lmat}) have the same structure), we know that $\widetilde{\S}_+(\widetilde{X},t)$ has an element of rank at most $\tw(G)+\varphi(\ell)+1.$ The completes the proof of Claim~\ref{claimStilde}.

With Claim~\ref{claimStilde}, for any $t\in [k],$ {consider} $Z^t\in \widetilde{\S}_+\big(\widetilde{X},t\big)$ such that $\rr\(Z^t\)\leq \tw(G)+\varphi(\ell)+1.$ Now, since we have found a low-rank element in each block, we move on to consider the chordal completion of $\widetilde{G}$ (see Figure~\ref{fig:tree decomposition completion} {in Appendix~\ref{app-illus}}). In detail, {let} 
\begin{equation}\label{whatGE}
\widehat{G}:=\big([n+k\ell],\widehat{\E}:=\cup_{t\in [k]}E\big(\widetilde{V}_t\big)\big),
\end{equation}
which is obtained by adding edges in each block $\widetilde{V}_t$ of $\widetilde{G}$ to make it a clique. From (\ref{edgesetblkt}), we can see that the newly added edges are $\sqcup_{j\in {\rm chd}(t)}E\(U_t,U_j\) \sqcup_{j_1,j_2\in {\rm chd}(t),j_1\neq j_2} E\(U_{j_1},U_{j_2}\),$ which corresponds to the ``$\ast$" blocks in (\ref{matSt}). It is easy to see that $\widehat{G}$ is a chordal graph with binary clique tree decomposition $\big(T,\widehat{\V}:=\big\{\widehat{V}_t:\ \widehat{V}_t=\widetilde{V}_t,\ t\in [k]\big\}\big).$ Consider the following set
\begin{equation}\label{setwhatS}
\widehat{\S}\big( \left\{Z^t:\ t\in [k]\right\},\widehat{G} \big):=\left\{ \widehat{Y}\in \S^{n+k\ell}_+:\ \forall t\in [k] ,\ \widehat{Y}_{\widehat{V}_t,\widehat{V}_t}=Z^t \right\}.
\end{equation}

\begin{claim}\label{nonempclaim}
$\widehat{\S}\big( \left\{Z^t:\ t\in [k]\right\},\widehat{G} \big)\neq \emptyset.$
\end{claim}
Consider an ordering $t_1,t_2,\ldots,t_k$ of $[k]$ such that the parents appear {\bf before} their children. Such an ordering can be found by breadth-first search. We recursively construct a matrix $\widehat{Z}\in \S^{n+k\ell}\big(\widehat{\E},0\big)$ as follows:
\begin{itemize}
\item [(1)] For any $i,j\in [n+k\ell]$ such that $i\neq j$ and $ij\notin \widehat{\E},$ let $\widehat{Z}_{ij}=\widehat{Z}_{ji}=0.$
\item [(2)] For $t=t_1,t_2,\ldots,t_k:$ let $\widehat{Z}_{\widehat{V}_t,\widehat{V}_t}:=Z^t.$ 
\end{itemize}
Note that in {item} (2) of the above construction, when we assign entries in the block $\widehat{Z}_{\widehat{V}_t,\widehat{V}_t},$ some entries of $\widehat{Z}_{\widehat{V}_t,\widehat{V}_t}$ might have been assigned before because different blocks may intersect. From the running intersection property of tree decomposition, those entries that have been assigned are in the block $\widehat{Z}_{\widehat{V}_p\cap\widehat{V}_t,\widehat{V}_p\cap\widehat{V}_t},$ where $p\in {\rm par}(t).$ From (\ref{blockeqv}), we know that 
\begin{equation}\label{intesectblk}
\widehat{V}_p\cap\widehat{V}_t= U_t\sqcup \(V_p\cap V_t\).
\end{equation}
From (\ref{edgesetblkt}) we have that
\begin{equation}\label{subsetedge}
E\( U_t\sqcup \(V_p\cap V_t\) \)\subset 
\E\big( \widetilde{G}\big[ \widetilde{V}_p \big] \big)\cap 
\E\big( \widetilde{G}\big[ \widetilde{V}_t \big] \big) .
\end{equation}
Because $Z^t\in \widetilde{\S}_+\big(\widetilde{X},t\big)$ and $Z^p\in \widetilde{\S}_+\big(\widetilde{X},p\big),$ from (\ref{setbfA}) and (\ref{subsetedge}), we have that
\begin{equation}\label{welldefiZ}
Z^t_{U_t\sqcup \(V_p\cap V_t\),U_t\sqcup \(V_p\cap V_t\)}=\widetilde{X}_{U_t\sqcup \(V_p\cap V_t\),U_t\sqcup \(V_p\cap V_t\)}=Z^p_{U_t\sqcup \(V_p\cap V_t\),U_t\sqcup \(V_p\cap V_t\)},
\end{equation}
which means that the assigned entries in $\widehat{Z}_{\widehat{V}_p\cap\widehat{V}_t,\widehat{V}_p\cap\widehat{V}_t}$ {based on $Z^p$ and $Z^t$} are {consistent}. Thus there is no conflict and hence $\widehat{Z}$ is well-defined. Because $\widehat{Z}\in \S^{n+k\ell}\big(\widehat{\E},0\big)$ and for any $t\in [k],$ $\widehat{Z}_{\widehat{V}_t,\widehat{V}_t}=Z^t,$ from Theorem 2.7 in \cite{jiang2017minimum}, we know that $\widehat{\S}\big( \left\{Z^t:\ t\in [k]\right\},\widehat{G} \big)\neq\emptyset.$ This completes the proof of Claim~\ref{nonempclaim}.

From Claim~\ref{nonempclaim}, we can use algorithm 3.1 in \cite{jiang2017minimum} to find $\widehat{Y}\in\widehat{\S}\big(\left\{Z^t:\ t\in [k]\right\},\widehat{G} \big) $ such that 
\begin{equation}\label{whatYcond1}
\rr\big(\widehat{Y}\big)\leq \max\left\{\rr(Z^t):\ t\in [k]\right\}\leq \tw(G)+\varphi(\ell)+1.
\end{equation}
For any $t\in [k],$ because $Z^t\in \widetilde{\S}_+\big(\widetilde{X},t\big)$ and $\widehat{Y}_{\widehat{V}_t,\widehat{V}_t}=Z^t,$ we have that $\widehat{Y}_{\widehat{V}_t,\widehat{V}_t}\in \widetilde{\S}_+\big(\widetilde{X},t\big).$ 
From (\ref{equivS}), we have that $\widehat{Y}\in \widetilde{\S}_+\big( \widetilde{X},\widetilde{G} \big).$ Let $Y:=\widehat{Y}_{1:n,1:n}.$ Then $Y$ is an optimal solution of (\ref{SDP}) of rank at most $\tw(G)+\varphi(\ell)+1.$
\end{proof}


\subsection{Proof of Proposition~\ref{propphi}}\label{Subsec:pfphi}

\begin{proof}

\noindent{\bf Step 1.} $\varphi(\ell)\leq \left\lfloor\(\sqrt{12\ell(\ell+1)+1}-1\)/2 \right\rfloor.$

\medskip

For any $U\in \R^{2\ell\times \ell}$ and $M\in \S^\ell.$ Suppose the set (\ref{cvxset}) is nonempty. Since there are $3\ell(\ell+1)/2$ linear constraints in (\ref{cvxset}), from the Barvinok-Pataki bound, we have that it has an element of rank at most $\left\lfloor\(\sqrt{12\ell(\ell+1)+1}-1\)/2 \right\rfloor.$

\medskip 

\noindent{\bf Step 2.} $\ell+1\leq \varphi(\ell).$

\medskip

Consider $U:=[I_\ell;I_\ell]$ and $M:=\dd ( [4\cdot {\bf 1}_{(\ell-1)\times 1};3]).$ It is easy the check that the following matrix is inside the set (\ref{cvxset}):
\begin{equation}\label{matelement}
\begin{bmatrix}
I_\ell & \dd ( [{\bf 1}_{(\ell-1)\times 1};1/2])\\
\dd ( [{\bf 1}_{(\ell-1)\times 1};1/2]) & I_\ell
\end{bmatrix},
\end{equation}
which implies that the set (\ref{cvxset}) is nonempty. For any $Y$ inside the set (\ref{cvxset}), it has the following matrix form 
\begin{equation}\label{matY222}
Y=\begin{bmatrix}
I_\ell & R^\top\\
R & I_\ell
\end{bmatrix},
\end{equation}
for some $R\in \R^{\ell\times \ell}.$ Because $Y\in \S^{2\ell}_+,$ we have that $\|R\|_2\leq 1,$ i.e., the largest singular value of $R$ is upper bounded by $1.$ From the constraint $U^\top Y U=M,$ we have that $R+R^\top=M-2I_\ell,$ which further implies that
\begin{equation}\label{matRskew}
R=\dd ( [{\bf 1}_{(\ell-1)\times 1};1/2])+\Omega,
\end{equation}
for some skew-symmetric matrix $\Omega\in \R^{\ell\times \ell}.$ If $\Omega\neq {\bf 0}_{\ell\times \ell},$ then a certain principle minor in $R$ will have the following form
\begin{equation}\label{Pmino}
\begin{bmatrix}
1&\beta\\-\beta& \alpha
\end{bmatrix},
\end{equation}
where $\alpha\in \{1/2,1\}$ and $\beta\neq 0.$ Simple calculations show that the singular values of the above matrix are
\begin{equation}\label{singvp}
\sqrt{\frac{2\beta^2+\alpha^2+1\pm \sqrt{(1-\alpha^2)^2+4\beta^2(\alpha-1)^2}}{2}},
\end{equation}
where the larger one is greater than $1.$ This implies that $\|R\|_2>1,$ which is a contradiction. Therefore, we have that $\Omega={\bf 0}_{\ell\times \ell}.$ From (\ref{matY222}) and  (\ref{matRskew}), we have that $Y$ is exactly the matrix (\ref{matelement}), whose rank is $\ell+1.$ This implies that the matrix (\ref{matelement}) is the unique element in the set (\ref{cvxset}) and hence $\varphi(\ell)\geq \ell+1.$ 
\end{proof}

\subsection{Proof of Theorem~\ref{lbthmtree}}\label{Subsec:pflb}

\begin{proof}
We first prove case {(i)}. For any $\omega,\ell\in \N$ and $n\geq \omega+\ell+1,$ from Lemma~\ref{omega=0}, there exists an SDP problem (\ref{exSDP}) with matrix dimension $\ell+1$ and $(G:=([\ell+1],\emptyset),\ell)-$SPLR structure. Moreover, any of its optimal solution has rank $\ell+1.$ By using Lemma~\ref{extlem}, we may extend this problem into an SDP problem of the form (\ref{extSDP}) with matrix dimension $n$ and $\big(\widehat{G},\ell\big)-$SPLR structure such that $\tw\big(\widehat{G}\big)=\pw\big(\widehat{G}\big)=\omega$ and any of its optimal solution has rank at least $\omega+\ell+1.$ Here we have used the fact that $\widehat{G}$ is the sparsity pattern {of a} block arrow matrix whose tree-width equals the path-width. Note that the corresponding notations in Lemma~\ref{extlem} should be $\sigma=\omega,$ $\hat{n}=n$ and $r=\ell+1.$ This completes the proof of case {(i)}.

Now, we {prove} the more complicated case {(ii)}. From Lemma~\ref{223}, there exists $M_1,M_2,M_3\in \S^\ell_+$ such that the following set is nonempty and any element has rank at least $\varphi(\ell):$ 
\begin{multline}\label{setBpro}
\B:=\big\{ X\in \S^{3\ell}_+:\ X_{1:\ell,1:\ell}=M_1,\ X_{\ell+1:2\ell,\ell+1:2\ell}=M_2,\ X_{2\ell+1:3\ell,2\ell+1:3\ell}=M_3,\\
 [I_\ell,I_\ell,-I_\ell]X[I_\ell;I_\ell;-I_\ell]={\bf 0}_{\ell\times \ell} \Big\}.
\end{multline}
Consider the following problem
\begin{multline}\label{exSDP6pro}
\min\Big\{ 0:\ X_{4\ell+1:5\ell,1:\ell}={\bf 0}_{\ell\times \ell},\ X_{4\ell+1:5\ell,2\ell+1:3\ell}={\bf 0}_{\ell\times \ell},\ X_{5\ell+1:6\ell,1:2\ell}={\bf 0}_{\ell\times 2\ell},\\
X_{3\ell+1:4\ell,3\ell+1:4\ell}=I_\ell+M_1,\ X_{4\ell+1:5\ell,4\ell+1:5\ell}=I_\ell+M_2,\  X_{5\ell+1:6\ell,5\ell+1:6\ell}=I_\ell+M_3,\\
X_{1:3\ell,1:3\ell}=I_{3\ell},\ X_{3\ell+1:4\ell,\ell+1:3\ell}={\bf 0}_{\ell\times 2\ell},\ V^\top X V={\bf 0}_{\ell\times \ell},\ X\in \S^{6\ell}_+\Big\},
\end{multline}
where $V:=[-I_\ell;-I_\ell;-I_\ell;I_\ell;I_\ell;-I_\ell]\in \R^{6\ell\times \ell}.$ All the affine constraints in (\ref{exSDP6pro}) can be written in the following matrix form
\begin{equation}\label{spXpro}
V^\top\(X=\begin{bmatrix}
I_\ell & {\bf 0}_{\ell\times \ell} & {\bf 0}_{\ell\times \ell} & Y_1^\top & {\bf 0}_{\ell\times \ell} & {\bf 0}_{\ell\times \ell}\\
 {\bf 0}_{\ell\times \ell}& I_\ell & {\bf 0}_{\ell\times \ell} & {\bf 0}_{\ell\times \ell} & Y_2^\top & {\bf 0}_{\ell\times \ell}\\
  {\bf 0}_{\ell\times \ell}& {\bf 0}_{\ell\times \ell} & I_\ell & {\bf 0}_{\ell\times \ell}  & {\bf 0}_{\ell\times \ell} & Y_3^\top\\
Y_1 & {\bf 0}_{\ell\times \ell} & {\bf 0}_{\ell\times \ell} & I_\ell+M_1 & Z_1^\top & Z_2^\top \\
{\bf 0}_{\ell\times \ell}  & Y_2 & {\bf 0}_{\ell\times \ell} & Z_1& I_\ell+M_2  & Z_3^\top \\
{\bf 0}_{\ell\times \ell}  & {\bf 0}_{\ell\times \ell}  & Y_3  & Z_2& Z_3  & I_\ell+M_3 \\
\end{bmatrix}\)V={\bf 0}_{\ell\times \ell},
\end{equation}
where $Y_1,Y_2,Y_3,Z_1,Z_2,Z_3\in \R^{\ell\times \ell}$ are the undetermined blocks. If  we consider the first 8 linear constraints in (\ref{exSDP6pro}) as the sparse part and the last one as the low-rank part, then problem $(\ref{exSDP6pro})$ has $(G,\ell)-$SPLR structure with the following aggregate sparsity pattern:
\begin{equation}\label{sparepatternG}
\begin{bmatrix}
* & * & * & \empty & * & *\\
*& * & * & * & \empty & *\\
*& * & * & *  & * & \empty\\
\empty & * & * & * & \empty & \empty \\
*  & \empty & * & \empty& *  & \empty \\
*  & *  & \empty  & \empty & \empty & * \\
\end{bmatrix},
\end{equation}
where $*$'s denote the blocks that appear in certain sparse constraints in (\ref{exSDP6pro}) while the empty parts denote the blocks that don't appear in any sparse constraints {and they correspond} to the variables $Y_i$'s, $Z_i$'s in (\ref{spXpro}). From (\ref{sparepatternG}), it is easy to see that $G$ is a chordal graph with tree-width $\tw(G)=3\ell-1.$ 

\medskip

\noindent{\bf Step 1.} The problem (\ref{exSDP6pro}) has an optimal solution.

\medskip

Consider $M\in \B$ such that
\begin{equation}\label{3b3pro}
M=\begin{bmatrix}
M_1&Z_1^\top & Z_2^\top\\
Z_1&M_2&Z_3^\top\\
Z_2&Z_3&M_3
\end{bmatrix}
\end{equation}
for some $Z_1,Z_2,Z_3\in \R^{\ell\times \ell}.$ Consider the following matrix
\begin{equation}\label{6lmatex}
X=\begin{bmatrix}
I_\ell & {\bf 0}_{\ell\times \ell} & {\bf 0}_{\ell\times \ell} & I_\ell & {\bf 0}_{\ell\times \ell} & {\bf 0}_{\ell\times \ell}\\
 {\bf 0}_{\ell\times \ell}& I_\ell & {\bf 0}_{\ell\times \ell} & {\bf 0}_{\ell\times \ell} & I_\ell & {\bf 0}_{\ell\times \ell}\\
  {\bf 0}_{\ell\times \ell}& {\bf 0}_{\ell\times \ell} & I_\ell & {\bf 0}_{\ell\times \ell}  & {\bf 0}_{\ell\times \ell} & -I_\ell\\
I_\ell & {\bf 0}_{\ell\times \ell} & {\bf 0}_{\ell\times \ell} & I_\ell+M_1 & Z_1^\top & Z_2^\top \\
{\bf 0}_{\ell\times \ell}  & I_\ell & {\bf 0}_{\ell\times \ell} & Z_1& I_\ell+M_2  & Z_3^\top \\
{\bf 0}_{\ell\times \ell}  & {\bf 0}_{\ell\times \ell}  & -I_\ell  & Z_2& Z_3  & I_\ell+M_3 \\
\end{bmatrix}.
\end{equation}
From (\ref{spXpro}), $X$ satisfies the sparse constraints in (\ref{exSDP6pro}). Note that the Schur complement of $X$ with respect to the first $3\ell$ by $3\ell$ block is exactly $M.$ Because $M\in \S^{3\ell}_+,$ we have that $X\in \S^{6\ell}_+.$ It remains to show that $V^\top X V={\bf 0}_{\ell\times \ell}.$ Because $M\in \B,$ we have that $M[I_\ell;I_\ell;-I_\ell]={\bf 0}_{3\ell\times \ell}.$ Combining this and the matrix form (\ref{6lmatex}), we can see that $XV={\bf 0}_{6\ell\times \ell}.$ Thus, $V^\top X V={\bf 0}_{\ell\times \ell}$ and hence $X$ is an optimal solution of (\ref{exSDP6pro}).

\medskip

\noindent{\bf Step 2.} Any optimal solution of the problem (\ref{exSDP6pro}) has rank at least $3\ell+\varphi(\ell).$ 

\medskip

Suppose $X$ is an optimal solution of (\ref{exSDP6pro}). From the constraints $V^\top X V={\bf 0}_{\ell\times \ell}$ and $X\in \S^{6\ell}_+,$ we have that $XV={\bf 0}_{6\ell\times \ell}.$ From the first $3\ell$ rows of the equation $XV={\bf 0}_{6\ell\times \ell},$ we have that in the matrix form (\ref{spXpro}), $Y_1=I_\ell,$ $Y_2=I_\ell$ and $Y_3=-I_\ell.$ Thus, $X$ still has the matrix form in (\ref{6lmatex}) in Step 1. From the constraint $V^\top X V={\bf 0}_{\ell\times \ell}$ and the matrix form (\ref{6lmatex}), we have that the Schur complement of $X$ with respect to 
 the first $3\ell$ by $3\ell$ block, which has the same matrix form as (\ref{3b3pro}), is inside $\B.$ Because any element of $\B$ has rank at least $\varphi(\ell),$ we have that $\rr\(X\)\geq 3\ell+\varphi(\ell).$

\medskip

\noindent{\bf Step 3.} Extending the problem (\ref{exSDP6pro}) into higher dimensional space.

From Step 2, we know that (\ref{exSDP6pro}) is an SDP problem with matrix dimension $6\ell$ and $(G,\ell)-$SPLR structure such that $\tw(G)=3\ell-1.$ Its optimal solution has rank at least $3\ell+\varphi(\ell).$ From Lemma~\ref{extlem}, we can see that for any $\omega\geq 3\ell-1$ and $n\geq \omega+3\ell+1,$ there exists an SDP problem (\ref{SDP}) with $\big(\widehat{G},\ell\big)-$SPLR structure such that $\tw\big(\widehat{G}\big)=\omega$ and any of its optimal solution has rank at least $\omega+\varphi(\ell)+1.$ Note that we apply Lemma~\ref{extlem} with $\sigma=\omega-3\ell+1,$ $\hat{n}=n$ and $r=3\ell+\varphi(\ell).$ This problem satisfies the conditions in Theorem~\ref{lbthmtree}. The proof is completed.
\end{proof}


\section{Conclusion}\label{Sec-conc}

In this study, we derive {a} unified framework called sparse extension to transfer {an} SDP problem with sparse plus low-rank data matrices into a sparse SDP problem with bounded tree-width. Based on this, we derive rank bounds for SDP problems with SPLR structure that are tight in the worst case. In the future, it would be interesting to design efficient algorithms for solving SDPs with SPLR structure and conduct {computational} complexity analysis.

\appendix

\section{Illustration of the sparse extension in Section~\ref{Sec:SET}}\label{app-illus}

In this subsection, we use some figures to illustrate the graph operations in the sparse extension in Section~\ref{Sec:SET}. For simplicity, we assume $k=7.$ 


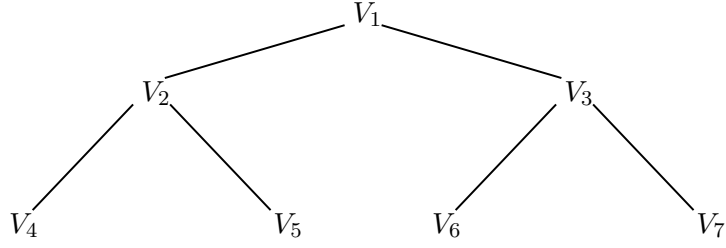
\begin{figure}[!h]      
\centering
\tikzset{every picture/.style={line width=0.75pt}} 

\begin{tikzpicture}[x=1pt,y=1pt,yscale=-1,xscale=1]
\draw (0,0) node [anchor=north west][inner sep=0.75pt]   [align=left] {$V_1$};
\draw    (-2,10) -- (-70,30) ; 
\draw    (12,10) -- (80,30) ; 

\draw (-80,30) node [anchor=north west][inner sep=0.75pt]   [align=left] {$V_2$};
\draw    (-82,40) -- (-120,80) ; 
\draw    (-68,40) -- (-30,80) ; 

\draw (80,30) node [anchor=north west][inner sep=0.75pt]   [align=left] {$V_3$};
\draw    (78,40) -- (40,80) ; 
\draw    (92,40) -- (130,80) ; 

\draw (-130,80) node [anchor=north west][inner sep=0.75pt]   [align=left] {$V_4$};

\draw (-30,80) node [anchor=north west][inner sep=0.75pt]   [align=left] {$V_5$};

\draw (30,80) node [anchor=north west][inner sep=0.75pt]   [align=left] {$V_6$};

\draw (130,80) node [anchor=north west][inner sep=0.75pt]   [align=left] {$V_7$};

\end{tikzpicture}
\caption{The binary clique tree decomposition of $G$ in Subsection~\ref{Subsec:spex}.}
\label{fig:tree decomposition lem}
\end{figure}

\begin{figure}[!h]      
\centering
\tikzset{every picture/.style={line width=0.75pt}} 

\begin{tikzpicture}[x=1pt,y=1pt,yscale=-1,xscale=1]
\draw (0,0) node [anchor=north west][inner sep=0.75pt]   [align=left] {$W_1=V_1$};
\draw    (-2,10) -- (-70,30) ; 
\draw    (12,10) -- (80,30) ; 

\draw (-90,30) node [anchor=north west][inner sep=0.75pt]   [align=left] {$W_2=V_2\setminus V_1$};
\draw    (-82,40) -- (-120,80) ; 
\draw    (-68,40) -- (-30,80) ; 

\draw (70,30) node [anchor=north west][inner sep=0.75pt]   [align=left] {$W_3=V_3\setminus V_1$};
\draw    (78,40) -- (40,80) ; 
\draw    (92,40) -- (130,80) ; 

\draw (-140,80) node [anchor=north west][inner sep=0.75pt]   [align=left] {$W_4=V_4\setminus V_2$};

\draw (-45,80) node [anchor=north west][inner sep=0.75pt]   [align=left] {$W_5=V_5\setminus V_2$};

\draw (33,80) node [anchor=north west][inner sep=0.75pt]   [align=left] {$W_6=V_6\setminus V_3$};

\draw (120,80) node [anchor=north west][inner sep=0.75pt]   [align=left] {$W_7=V_7\setminus V_3$};

\end{tikzpicture}
\caption{({\bf Step 1} in Subsection~\ref{Subsec:spex}) Partition of $[n]$ according to the tree decomposition of $G$.}
\label{fig:tree decomposition partition}
\end{figure}
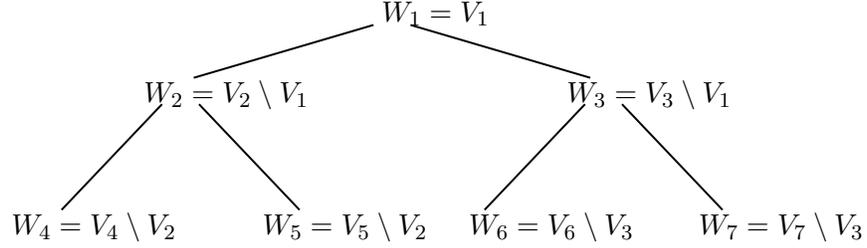

\begin{figure}[!h]      
\centering
\tikzset{every picture/.style={line width=0.75pt}} 

\begin{tikzpicture}[x=1pt,y=1pt,yscale=-1,xscale=1]
\draw (0,0) node [anchor=north west][inner sep=0.75pt]   [align=left] {$V_1$};
\draw    (-2,10) -- (-70,50) ; 
\draw    (12,10) -- (80,50) ; 
\draw (0,-40) node [anchor=north west][inner sep=0.75pt]   [align=left] {$U_1$};
\draw (4,-2) -- (4,-28); 
\draw (6,-2) -- (6,-28); 

\draw (-65,3)--(0,3); 
\draw (-65,5)--(0,5); 

\draw (12,3)--(77,3); 
\draw (12,5)--(77,5); 

\draw (-80,50) node [anchor=north west][inner sep=0.75pt]   [align=left] {$V_2$};
\draw    (-82,60) -- (-120,100) ; 
\draw    (-68,60) -- (-30,100) ; 
\draw (-80,53) -- (-115,53);
\draw (-80,55) -- (-115,55);
\draw (-30,53) -- (-65,53);
\draw (-30,55) -- (-65,55);

\draw (-80,0) node [anchor=north west][inner sep=0.75pt]   [align=left] {$U_2$};
\draw (-74,48) -- (-74,12); 
\draw (-76,48) -- (-76,12); 

\draw (80,50) node [anchor=north west][inner sep=0.75pt]   [align=left] {$V_3$};
\draw    (78,60) -- (40,100) ; 
\draw    (92,60) -- (130,100) ; 
\draw (80,0) node [anchor=north west][inner sep=0.75pt]   [align=left] {$U_3$};
\draw (84,48) -- (84,12); 
\draw (86,48) -- (86,12); 
\draw (45,53) -- (80,53);
\draw (45,55) -- (80,55);
\draw (95,53) -- (130,53);
\draw (95,55) -- (130,55);

\draw (-130,100) node [anchor=north west][inner sep=0.75pt]   [align=left] {$V_4$};
\draw (-130,50) node [anchor=north west][inner sep=0.75pt]   [align=left] {$U_4$};
\draw (-124,98) -- (-124,62); 
\draw (-126,98) -- (-126,62); 

\draw (-30,100) node [anchor=north west][inner sep=0.75pt]   [align=left] {$V_5$};
\draw (-30,50) node [anchor=north west][inner sep=0.75pt]   [align=left] {$U_5$};
\draw (-24,98) -- (-24,62); 
\draw (-26,98) -- (-26,62); 

\draw (30,100) node [anchor=north west][inner sep=0.75pt]   [align=left] {$V_6$};
\draw (30,50) node [anchor=north west][inner sep=0.75pt]   [align=left] {$U_6$};
\draw (36,98) -- (36,62); 
\draw (34,98) -- (34,62); 

\draw (130,100) node [anchor=north west][inner sep=0.75pt]   [align=left] {$V_7$};
\draw (130,50) node [anchor=north west][inner sep=0.75pt]   [align=left] {$U_7$};
\draw (136,98) -- (136,62); 
\draw (134,98) -- (134,62); 

\end{tikzpicture}
\caption{({\bf Step 2} in Subsection~\ref{Subsec:spex}) Extension $G$ to $\widetilde{G}$, where $\widetilde{G}$ may not be a chordal graph.}
\label{fig:tree extension}
\end{figure}
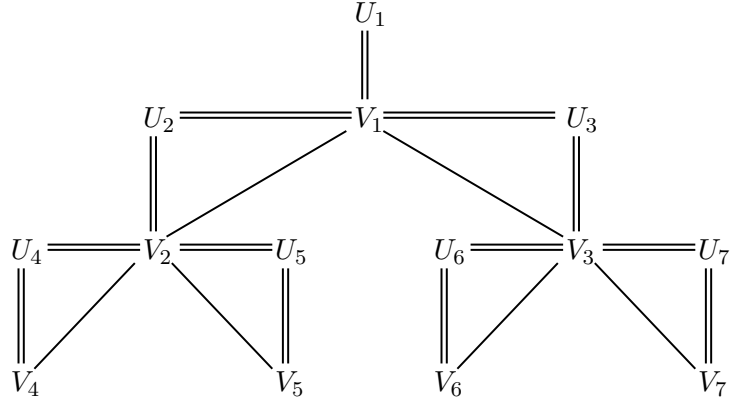

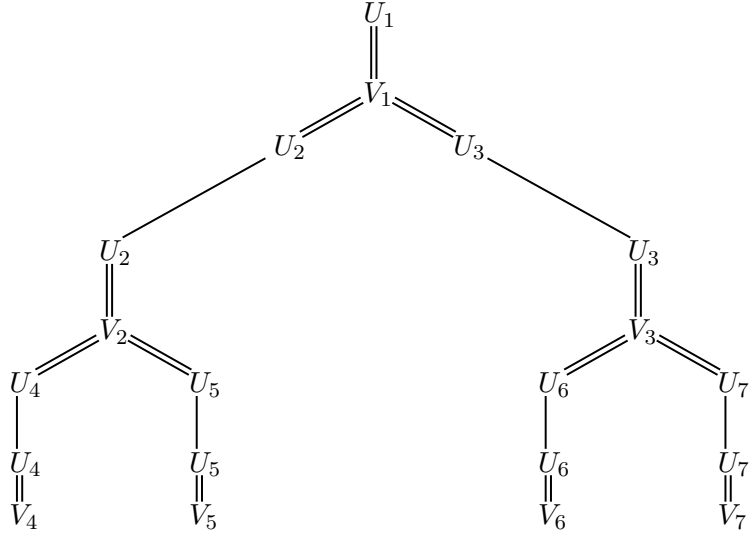
\begin{figure}[!h]      
\centering
\tikzset{every picture/.style={line width=0.75pt}} 

\begin{tikzpicture}[x=1pt,y=1pt,yscale=-1,xscale=1]
\draw (0,0) node [anchor=north west][inner sep=0.75pt]   [align=left] {$V_1$};
\draw (0,-30) node [anchor=north west][inner sep=0.75pt]   [align=left] {$U_1$};
\draw (-34,20) node [anchor=north west][inner sep=0.75pt]   [align=left] {$U_2$};
\draw (34,20) node [anchor=north west][inner sep=0.75pt]   [align=left] {$U_3$};
\draw (4,0)--(4,-20); 
\draw (6,0)--(6,-20); 

\draw (0,7)--(-23,20); 
\draw (1,9)--(-22,22); 

\draw (13,7)--(36,20); 
\draw (12,9)--(35,22); 


\draw (-36,30)--(-90,60); 


\draw (48,30)--(102,60); 

\draw (-100,90) node [anchor=north west][inner sep=0.75pt]   [align=left] {$V_2$};
\draw (-100,60) node [anchor=north west][inner sep=0.75pt]   [align=left] {$U_2$};
\draw (-134,110) node [anchor=north west][inner sep=0.75pt]   [align=left] {$U_4$};
\draw (-66,110) node [anchor=north west][inner sep=0.75pt]   [align=left] {$U_5$};
\draw (-94,90)--(-94,70); 
\draw (-96,90)--(-96,70); 
\draw (-100,97)--(-123,110); 
\draw (-99,99)--(-122,112); 
\draw (-87,97)--(-64,110); 
\draw (-88,99)--(-65,112); 


\draw (-130,120)--(-130,140); 


\draw (-62,120)--(-62,140);

\draw (100,90) node [anchor=north west][inner sep=0.75pt]   [align=left] {$V_3$};
\draw (100,60) node [anchor=north west][inner sep=0.75pt]   [align=left] {$U_3$};
\draw (66,110) node [anchor=north west][inner sep=0.75pt]   [align=left] {$U_6$};
\draw (134,110) node [anchor=north west][inner sep=0.75pt]   [align=left] {$U_7$};
\draw (104,90)--(104,70); 
\draw (106,90)--(106,70); 
\draw (100,97)--(77,110); 
\draw (101,99)--(78,112); 
\draw (113,97)--(136,110); 
\draw (112,99)--(135,112); 


\draw (70,120)--(70,140); 


\draw (138,120)--(138,140); 

\draw (-134,140) node [anchor=north west][inner sep=0.75pt]   [align=left] {$U_4$};
\draw (-134,160) node [anchor=north west][inner sep=0.75pt]   [align=left] {$V_4$};
\draw (-128,150)--(-128,160); 
\draw (-130,150)--(-130,160); 

\draw (-66,140) node [anchor=north west][inner sep=0.75pt]   [align=left] {$U_5$};
\draw (-66,160) node [anchor=north west][inner sep=0.75pt]   [align=left] {$V_5$};
\draw (-60,150)--(-60,160); 
\draw (-62,150)--(-62,160); 

\draw (66,140) node [anchor=north west][inner sep=0.75pt]   [align=left] {$U_6$};
\draw (66,160) node [anchor=north west][inner sep=0.75pt]   [align=left] {$V_6$};
\draw (70,150)--(70,160); 
\draw (72,150)--(72,160); 

\draw (134,140) node [anchor=north west][inner sep=0.75pt]   [align=left] {$U_7$};
\draw (134,160) node [anchor=north west][inner sep=0.75pt]   [align=left] {$V_7$};
\draw (138,150)--(138,160); 
\draw (140,150)--(140,160); 

\end{tikzpicture}
\caption{({\bf Step 2} in Subsection~\ref{Subsec:spex}) The binary tree decomposition $\big(T,\widetilde{\V}\big)$ of $\widetilde{G}.$}
\label{fig:tree decomposition extention}
\end{figure}

\begin{figure}[!h]      
\centering
\tikzset{every picture/.style={line width=0.75pt}} 

\begin{tikzpicture}[x=1pt,y=1pt,yscale=-1,xscale=1]
\draw (0,0) node [anchor=north west][inner sep=0.75pt]   [align=left] {$V_1$};
\draw (0,-30) node [anchor=north west][inner sep=0.75pt]   [align=left] {$U_1$};
\draw (-34,20) node [anchor=north west][inner sep=0.75pt]   [align=left] {$U_2$};
\draw (34,20) node [anchor=north west][inner sep=0.75pt]   [align=left] {$U_3$};
\draw (4,0)--(4,-20); 
\draw (6,0)--(6,-20); 

\draw (0,7)--(-23,20); 
\draw (1,9)--(-22,22); 

\draw (13,7)--(36,20); 
\draw (12,9)--(35,22); 

\draw (0,-26)--(-31,20); 
\draw (0,-21)--(-28,20); 

\draw (13,-26)--(41,20); 
\draw (13,-21)--(38,20); 

\draw (-22,23)--(34,23); 
\draw (-22,25)--(34,25); 


\draw (-36,30)--(-90,70); 


\draw (48,30)--(102,70);

\draw (-100,100) node [anchor=north west][inner sep=0.75pt]   [align=left] {$V_2$};
\draw (-100,70) node [anchor=north west][inner sep=0.75pt]   [align=left] {$U_2$};
\draw (-134,120) node [anchor=north west][inner sep=0.75pt]   [align=left] {$U_4$};
\draw (-66,120) node [anchor=north west][inner sep=0.75pt]   [align=left] {$U_5$};
\draw (-94,100)--(-94,80); 
\draw (-96,100)--(-96,80); 
\draw (-100,107)--(-123,120); 
\draw (-99,109)--(-122,122); 
\draw (-87,107)--(-64,120); 
\draw (-88,109)--(-65,122); 

\draw (-100,75)--(-131,121); 
\draw (-100,80)--(-128,121); 

\draw (-86,74)--(-59,120); 
\draw (-86,79)--(-62,120); 

\draw (-122,123)--(-66,123); 
\draw (-122,125)--(-66,125); 


\draw (-130,130)--(-130,150); 


\draw (-62,130)--(-62,150);

\draw (100,100) node [anchor=north west][inner sep=0.75pt]   [align=left] {$V_3$};
\draw (100,70) node [anchor=north west][inner sep=0.75pt]   [align=left] {$U_3$};
\draw (66,120) node [anchor=north west][inner sep=0.75pt]   [align=left] {$U_6$};
\draw (134,120) node [anchor=north west][inner sep=0.75pt]   [align=left] {$U_7$};
\draw (104,100)--(104,80); 
\draw (106,100)--(106,80); 
\draw (100,107)--(77,120); 
\draw (101,109)--(78,122); 
\draw (113,107)--(136,120); 
\draw (112,109)--(135,122); 

\draw (100,74)--(69,120); 
\draw (100,79)--(72,120); 

\draw (114,74)--(141,120); 
\draw (114,79)--(138,120); 

\draw (78,123)--(134,123); 
\draw (78,125)--(134,125); 


\draw (70,130)--(70,150); 


\draw (138,130)--(138,150); 

\draw (-134,150) node [anchor=north west][inner sep=0.75pt]   [align=left] {$U_4$};
\draw (-134,170) node [anchor=north west][inner sep=0.75pt]   [align=left] {$V_4$};
\draw (-128,160)--(-128,170); 
\draw (-130,160)--(-130,170); 

\draw (-66,150) node [anchor=north west][inner sep=0.75pt]   [align=left] {$U_5$};
\draw (-66,170) node [anchor=north west][inner sep=0.75pt]   [align=left] {$V_5$};
\draw (-60,160)--(-60,170); 
\draw (-62,160)--(-62,170); 

\draw (66,150) node [anchor=north west][inner sep=0.75pt]   [align=left] {$U_6$};
\draw (66,170) node [anchor=north west][inner sep=0.75pt]   [align=left] {$V_6$};
\draw (70,160)--(70,170); 
\draw (72,160)--(72,170); 

\draw (134,150) node [anchor=north west][inner sep=0.75pt]   [align=left] {$U_7$};
\draw (134,170) node [anchor=north west][inner sep=0.75pt]   [align=left] {$V_7$};
\draw (138,160)--(138,170); 
\draw (140,160)--(140,170); 

\end{tikzpicture}
\caption{(Proof of Theorem~\ref{ubthmtree}) The binary clique tree decomposition 
$\big(T,\widehat{\V}\big)$ of the chordal completion graph $\widehat{G}$.}
\label{fig:tree decomposition completion}
\end{figure}
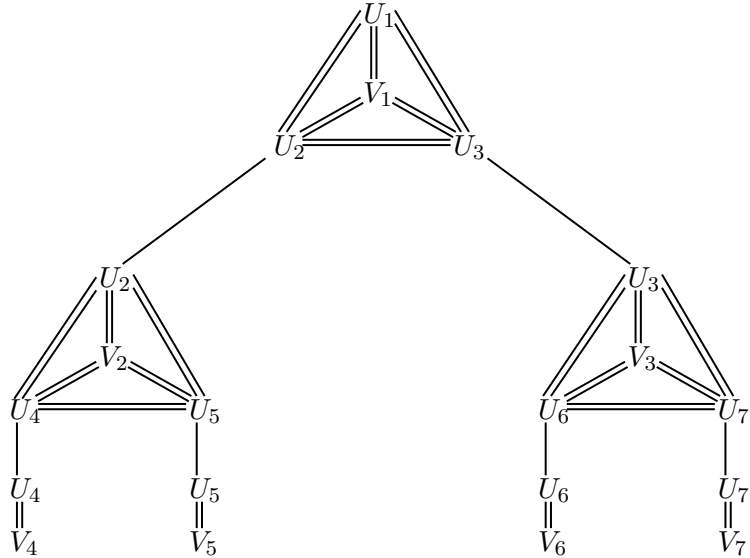

\section{Some results needed in Subsection~\ref{Subsec:pfthm} and \ref{Subsec:pfphi}}\label{app-useful}

\begin{lem}\label{223}
For any $\ell\in \N^+,$ $\varphi(\ell)$ is the smallest integer such that for any $M_1,M_2,M_3\in \S^\ell_+,$ the following set
\begin{multline}\label{lemsetB}
\B:=\Big\{ X\in \S^{3\ell}_+:\ X_{1:\ell,1:\ell}=M_1,\ X_{\ell+1:2\ell,\ell+1:2\ell}=M_2,\ X_{2\ell+1:3\ell,2\ell+1:3\ell}=M_3,\\
 [I_\ell,I_\ell,-I_\ell]X[I_\ell;I_\ell;-I_\ell]={\bf 0}_{\ell\times \ell} \Big\}.
\end{multline}
is either an empty set or has an element of rank at most $\varphi(\ell),$ {where $\varphi(\ell)$ is given in Definition~\ref{cvxnb}.} 
\end{lem}

\begin{proof}

\noindent{\bf Step 1.} Suppose $\B\neq \emptyset.$ Then $\B$ has an element of rank at most $\varphi(\ell).$

\medskip

Let $X\in \B$ be an element of rank $r.$ If $r\leq \varphi(\ell),$ then we have that $\rr(X)\leq \varphi(\ell).$ Thus, we suppose that $r\geq \varphi(\ell)+1.$ $X$ has decomposition $X=RR^\top$ such that $R\in \R^{3\ell\times r}$ and $R=[R_1;R_2;R_3]$ {satisfies that}
\begin{equation}\label{condR223}
\forall i\in [3],\ R_i\in \R^{\ell\times r}, \;\; R_iR_i^\top=M_i,\ R_1+R_2=R_3.
\end{equation}
From (\ref{condR223}), we have that $(R_1+R_2)(R_1+R_2)^\top=R_3R_3^\top,$ which implies that
\begin{equation}\label{condR123}
M_1+M_2+R_1R_2^\top+R_2R_1^\top=M_3.
\end{equation}
Because $r\geq \varphi(\ell)+1>\ell,$ by considering the reduced QR decompositions, there exists $R_1',R_2',R_3'\in \R^{\ell\times \ell}$ and $Q_1,Q_2,Q_3\in \R^{\ell\times r}$ such that for any $i\in [3],$ $Q_iQ_i^\top=I_\ell$ and $R_i'Q_i=R_i.$ Substitute this into (\ref{condR123}), we get
\begin{equation}\label{condR'123}
R_1'{R_1'}^\top+R_2'{R_2'}^\top+R_1'Q_1Q_2^\top {R_2'}^\top+R_2'Q_2Q_1^\top {R_1'}^\top=M_3,
\end{equation}
which can be reformulated as
\begin{equation}\label{condR'123}
[R_1',R_2']\begin{bmatrix} I_\ell & Q_1Q_2^\top \\ Q_2Q_1^\top & I_\ell \end{bmatrix}[R_1',R_2']^\top=M_3.
\end{equation}
From Definition~\ref{cvxnb}, there exists $W\in \S^{2\ell}_+$ such that $\ell\leq \rr(W)=:r'\leq \varphi(\ell)$ and
\begin{equation}\label{condW22}
[R_1',R_2']\(W=\begin{bmatrix} I_\ell & W_1^\top \\ W_1 & I_\ell \end{bmatrix}\)[R_1',R_2']^\top=M_3
\end{equation}
Then, we can decompose $W=[Q_1';Q_2'][Q_1';Q_2']^\top$ such that $Q_1',Q_2'\in \R^{\ell\times r'}$ and $Q_1'{Q_1'}^\top=Q_2'{Q_2'}^\top=I_\ell.$ Substitute this into (\ref{condW22}) and using the fact that $M_3=R_3'{R_3'}^\top,$ we have that there exists $Q_3'\in \R^{\ell\times r'}$ such that
\begin{equation}\label{condW'22}
R_1'Q_1'+R_2'Q_2'=R_3'Q_3'.
\end{equation}
Define $R':=\left[R_1'Q_1';R_2'Q_2';R_3'Q_3'\right]\in \R^{3\ell\times r'}$ and $X':=R'{R'}^\top.$ From (\ref{condW'22}), we know that $X'\in \B$ and $\rr(X')\leq r'\leq \varphi(\ell).$ 

\bigskip

\noindent{\bf Step 2.} There exists $M_1,M_2,M_3\in \S^\ell$ such that $\B\neq \emptyset$ and any element of $\B$ has rank at least $\varphi(\ell).$ 

\medskip

From Definition~\ref{cvxnb}, we know that there exists $U_1,U_2\in \R^{\ell\times \ell}$ and $M_3\in \S^\ell$ such that the following set is nonempty and any of its element has rank at least $\varphi(\ell)$:
\begin{equation}\label{cvxsetl}
\C:=\left\{ Y\in \S^{2\ell}_+:\ [U_1,U_2] Y[U_1,U_2]^\top=M_3,\ Y_{1:\ell,1:\ell}=I_\ell,\ Y_{\ell+1:2\ell,\ell+1:2\ell}=I_\ell \right\}.
\end{equation}
Let $M_1:=U_1U_1^\top,$ $M_2:=U_2U_2^\top$ and $M_3=U_3U_3^\top$ for some $U_3\in \R^{\ell\times \ell}.$ Suppose $Y\in \C,$ then it has decomposition 
{$Y=[Q_1;Q_2][Q_1;Q_2]^\top$ with $[Q_1;Q_2]\in \R^{2\ell\times r}$} such that $r\geq \varphi(\ell)$ and for any $i\in [2],$ $Q_iQ_i^\top=I_\ell.$ Because $[U_1,U_2] Y[U_1,U_2]^\top=U_3U_3^\top,$ we have that there exists $Q_3\in \R^{\ell\times r}$ such that $Q_3Q_3^\top=I_\ell$ and 
\begin{equation}\label{condU123}
U_1Q_1+U_2Q_2=U_3Q_3.
\end{equation}
Let $R:=[U_1Q_1;U_2Q_2;U_3Q_3]$ and $X:=RR^\top.$ From (\ref{condU123}), we can see that $X\in \B$ for the above defined $M_1,M_2,M_3.$ Now, we prove that every element of $\B$ has rank at least $\varphi(\ell).$
 Assume on the contrary that $\B$ has an element $X'$ such that $\rr(X')=r'\leq \varphi(\ell)-1.$ Let $X'=R'{R'}^\top$ such that $R'=[R'_1;R'_2;R'_3]$ and for any $i\in [3],$ $R'_i\in \R^{\ell\times r'}.$ We have that
\begin{equation}\label{R'R}
\forall i\in [3],\ R_i'{R_i'}^\top=M_i=U_iU_i^\top,\ R_1'+R_2'=R_3'.
\end{equation}
From (\ref{R'R}), we can see that for any $i\in [3],$ if $r'\geq \ell,$ then there exists $Q_i'\in \R^{\ell\times r'}$ such that $U_iQ_i'=R_i'$ and $Q_i'Q_i'^\top = I_\ell$. If $r'<\ell,$ then there exists $Q_i'\in \R^{\ell\times \ell}$ such that $U_iQ_i'=\left[ R_i',{\bf 0}_{\ell\times (\ell-r')} \right]$ and $Q_i'Q_i'^\top = I_\ell$.
In each case, we have find a set of matrices $Q_1',Q_2',Q_3'\in \R^{\ell\times r''}$ such that $\ell\leq r''\leq \varphi(\ell)-1$\footnote{Note that we have used the fact that $\varphi(\ell)\geq \ell+1$ in Proposition~\ref{propphi}, which was proved independently.} and
\begin{equation}\label{condQ'22}
\forall i\in [3],\ Q_i'{Q_i'}^\top=I_\ell,\ U_1Q_1'+U_2Q_2'=U_3Q_3'.
\end{equation}
Let $H:=[Q_1';Q_2']\in \R^{2\ell\times r''}$ and $Y':=HH^\top.$ From (\ref{condQ'22}), we have that $Y'\in \C.$ Because $\rr(Y')=r''\leq \varphi(\ell)-1,$ we get a contradiction to that any element of $\C$ has rank at least $\varphi(\ell).$ 
\end{proof}

\begin{lem}\label{rankreduc}
For any $p,\ell\in \N,$ consider $U:=[V;I_\ell;I_\ell;-I_\ell]\in \R^{(p+3\ell)\times \ell}$ such that $V\in \R^{p\times \ell},$ $B\in \S^p,$ $C\in \R^{3\ell\times p}$ and $M_1,M_2,M_3\in \S^\ell.$ Then the following set is either empty or has an element of rank at most $p+\varphi(\ell).$
\begin{multline}\label{set3lX}
\B:=\Big\{X\in \S^{p+3\ell}_+:\ X_{1:p,1:p}=B,\ X_{p+1:p+3\ell,1:p}=C,\
 X_{p+1:p+\ell,p+1:p+\ell}=M_1,\\ X_{p+\ell+1:p+2\ell,p+\ell+1:p+2\ell}=M_2,\ X_{p+2\ell+1:p+3\ell,p+2\ell+1:p+3\ell}=M_3,\ U^\top X U={\bf 0}_{\ell\times \ell}\Big\}.
\end{multline}
\end{lem}

\begin{proof}
We assume that $\ell\in \N^+$ because otherwise $\ell=0,$ $\varphi(\ell)=0$ and there is nothing to prove. The linear constraints in $\B$ can be written as follows

\begin{equation}\label{3lmat}
[V^\top,I_\ell,I_\ell,-I_\ell]\(X=\begin{bmatrix}
    \begin{array}{c|c}
        B
        &
        C^\top
        \\
    \hline
    C
    &
    \begin{array}{ccc}
        M_1 & * & * \\
        * & M_2 & * \\
        * & * & M_3
    \end{array}
\end{array}
\end{bmatrix}\)
\begin{bmatrix}
V\\
I_\ell\\
I_\ell\\
-I_\ell
\end{bmatrix}={\bf 0}_{\ell\times \ell},
\end{equation}
where $*$ denotes the blocks that are undetermined in the sparse linear constraints in $\B.$ Suppose $\B\neq \emptyset,$ then it has an element $Y$ such that

\begin{equation}\label{3lmatY}
Y=\begin{bmatrix}
    \begin{array}{c|c}
        B
        &
        C^\top
        \\
    \hline
    C
    &
    \begin{array}{ccc}
        M_1 & Y_1^\top & Y_2^\top \\
        Y_1 & M_2 & Y_3^\top \\
        Y_2 & Y_3 & M_3
    \end{array}
\end{array}
\end{bmatrix}\in \S^{p+3\ell}_+,\quad U^\top YU={\bf 0}_{\ell\times \ell}
\end{equation}
for some matrices $Y_1,Y_2,Y_3\in \R^{\ell\times \ell}.$ Define $L\in \R^{(p+3\ell)\times (p+3\ell)}$ as follows:
\begin{equation}\label{Lmat}
L:=\begin{bmatrix}
I_p&{\bf 0}_{p\times 3\ell}\\
-CB^{\dagger}&I_{3\ell}
\end{bmatrix},
\end{equation}
where $B^{\dagger}$ is the pseudo inverse of $B.$ Let $Y'=LYL^\top.$ We have that $Y'$ satisfies the follow conditions
\begin{equation}\label{3lmatY'}
Y'=\begin{bmatrix}
    \begin{array}{c|c}
        B
        &
        {\bf 0}_{p\times 3\ell}
        \\
    \hline
    {\bf 0}_{3\ell\times p}
    &
    \begin{array}{ccc}
        M_1' & {Y_1'}^\top & {Y_2'}^\top \\
        Y_1' & M_2' & {Y_3'}^\top \\
        Y_2' & Y_3' & M_3'
    \end{array}
\end{array}
\end{bmatrix}\in \S^{p+3\ell}_+,\quad U^\top L^{-1}Y'L^{-\top}U={\bf 0}_{\ell\times \ell},
\end{equation}
where $Y'_1,Y'_2,Y'_3\in \R^{\ell\times \ell}$ and
\begin{equation}\label{Y'block}
M':=\begin{bmatrix}
\begin{array}{ccc}
        M_1' & {Y_1'}^\top & {Y_2'}^\top \\
        Y_1' & M_2' & {Y_3'}^\top \\
        Y_2' & Y_3' & M_3'
    \end{array}
    \end{bmatrix}=
    \begin{bmatrix}
     \begin{array}{ccc}
        M_1 & Y_1^\top & Y_2^\top \\
        Y_1 & M_2 & Y_3^\top \\
        Y_2 & Y_3 & M_3
    \end{array}
    \end{bmatrix}-CB^{\dagger}C^\top\in \S^{3\ell}_+
\end{equation}
is the generalized Schur complement of $B$ in $Y$ and $L^{-\top}U=[V';I_\ell;I_\ell;-I_\ell]$ such that $V'=V+B^{\dagger}C^\top[I_\ell;I_\ell;-I_\ell].$ From (\ref{3lmatY'}), we have that
\begin{equation}\label{condBcl}
{V'}^\top B {V'}={\bf 0}_{\ell\times \ell},\quad [I_\ell,I_\ell,-I_\ell]M'[I_\ell;I_\ell;-I_\ell]={\bf 0}_{\ell\times \ell}.
\end{equation}
Consider the following set
\begin{multline}\label{setC}
\B':=\Big\{ X\in \S^{3\ell}_+:\ X_{1:\ell,1:\ell}=M_1',\ X_{\ell+1:2\ell,\ell+1:2\ell}=M_2',\ X_{2\ell+1:3\ell,2\ell+1:3\ell}=M_3',\\
 [I_\ell,I_\ell,-I_\ell]X[I_\ell;I_\ell;-I_\ell]={\bf 0}_{\ell\times \ell} \Big\}.
\end{multline}
From (\ref{Y'block}) and (\ref{condBcl}), we have that $M'\in \B'.$ From Lemma~\ref{223}, we know that $\B'$ contains an element of rank at most $\varphi(\ell).$ Let $M''\in \B'$ such that $\rr(M'')\leq \varphi(\ell).$ Consider the following matrix
\begin{equation}\label{matY''}
Y'':=\begin{bmatrix}
B&{\bf 0}_{p\times 3\ell}\\
{\bf 0}_{3\ell\times p} & M''
\end{bmatrix}=
\begin{bmatrix}
    \begin{array}{c|c}
        B
        &
        {\bf 0}_{p\times 3\ell}
        \\
    \hline
    {\bf 0}_{3\ell\times p}
    &
    \begin{array}{ccc}
        M_1' & {Y_1''}^\top & {Y_2''}^\top \\
        Y_1'' & M_2' & {Y_3''}^\top \\
        Y_2'' & Y_3'' & M_3'
    \end{array}
\end{array}
\end{bmatrix}\in \S^{p+3\ell}_+,
\end{equation}
where $Y_1'',Y_2'',Y_3''\in \R^{\ell\times \ell}.$ Because $\rr(M'')\leq \varphi(\ell),$ we have that 
\begin{equation}\label{rankY''}
\rr(Y'')=\rr(B)+\rr(M'')\leq p+\varphi(\ell).
\end{equation}
From (\ref{condBcl}) and (\ref{setC}), we can see that 
\begin{equation}\label{afconY''}
U^\top L^{-1}Y''L^{-\top} U={\bf 0}_{\ell\times \ell}.
\end{equation}
Moreover, from (\ref{Lmat}) and (\ref{Y'block}), $L^{-1}Y''L^{-\top}$ has the following matrix form
\begin{equation}\label{matLY''L}
L^{-1}Y''L^{-\top}=\begin{bmatrix}
    \begin{array}{c|c}
        B
        &
        C^\top
        \\
    \hline
    C
    &
    \begin{array}{ccc}
        M_1 & * & * \\
        * & M_2 & * \\
        * & * & M_3
    \end{array}
\end{array}
\end{bmatrix}
\end{equation}
From (\ref{rankY''}), (\ref{afconY''}) and (\ref{matLY''L}), we have that $L^{-1}Y''L^{-\top}$ is an element of $\B$ with rank at most $p+\varphi(\ell).$
\end{proof}

\section{Some results needed in Subsection~\ref{Subsec:pflb}}\label{app-useful-1}


\begin{lem}\label{omega=0}
For any $\ell\in \N,$ consider the following SDP problem
\begin{equation}\label{exSDP}
\min\left\{ 0:\ \dd(X)={\bf 1}_{\ell+1},\ A^\top X A= I_{\ell}+{\bf 1}_{\ell}{\bf 1}_{\ell}^\top,\ X\in \S^{\ell+1}_+ \right\},
\end{equation}
where $A=\left[ I_{\ell};{\bf 1}_{\ell}^\top \right].$ Then problem (\ref{exSDP}) has $\(G:=\([\ell+1],\emptyset\),\ell\)-$SPLR structure. Moreover, any optimal solution of (\ref{exSDP}) has rank $\ell+1.$
\end{lem}

\begin{proof}
We only consider $\ell\geq 1,$ since there is nothing to prove when $\ell=0.$ It is easy to see that problem (\ref{exSDP}) has $(\([\ell+1],\emptyset\),\ell)-$SPLR structure, where the sparse constraints are $\dd(X)={\bf 1}_{\ell+1}$ and the low-rank constraints are $A^\top X A= I_{\ell}+{\bf 1}_{\ell}{\bf 1}_{\ell}^\top.$ Moreover, the matrix $X=I_{\ell+1}$ is an optimal solution of rank $\ell+1.$ Now, assume on the contrary that the problem (\ref{exSDP}) has an optimal solution $X$ of rank $r\leq \ell.$ Then it has the decomposition $X=RR^\top$ for some $R\in \R^{(\ell+1)\times r}.$ From the diagonal constraint, we have that $\dd(RR^\top)={\bf 1}_{\ell+1}$ and hence $\|R_{\ell+1,:}\|=1.$ From $A^\top X A= I_{\ell}+{\bf 1}_{\ell}{\bf 1}_{\ell}^\top$, we have that
\begin{equation}\label{Feb_5_1}
R_{1:\ell,:}R_{1:\ell,:}^\top+{\bf 1}_{\ell}R_{\ell+1,:}R_{1:\ell,:}^\top+R_{1:\ell,:}R_{\ell+1,:}^\top {\bf 1}_{\ell}^\top+{\bf 1}_{\ell}{\bf 1}_{\ell}^\top=I_{\ell}+{\bf 1}_{\ell}{\bf 1}_{\ell}^\top,
\end{equation}
which can be simplified to 
\begin{equation}\label{Feb_5_2}
R_{1:\ell,:}R_{1:\ell,:}^\top+{\bf 1}_{\ell}R_{\ell+1,:}R_{1:\ell,:}^\top+R_{1:\ell,:}R_{\ell+1,:}^\top {\bf 1}_{\ell}^\top=I_{\ell}.
\end{equation}
Because $\dd(R_{1:\ell,:}R_{1:\ell,:}^\top)={\bf 1}_\ell,$ from the diagonal entries of (\ref{Feb_5_2}), we have that $R_{1:\ell,:}R_{\ell+1,:}^\top={\bf 0}_{\ell}.$ Substitute this into (\ref{Feb_5_2}), we get $R_{1:\ell,:}R_{1:\ell,:}^\top=I_\ell.$ Because $r\leq \ell,$ we have that $R_{1:\ell,:}$ is full rank and $r=\ell.$ This together with $R_{1:\ell,:}R_{\ell+1,:}^\top={\bf 0}_{\ell\times 1}$ implies that $R_{\ell+1,:}={\bf 0}_{1\times r},$ which is contradictory to the fact that $\|R_{\ell+1,:}\|=1.$ This completes the proof. 
\end{proof}

\begin{lem}\label{extlem}
For any $G=([n],\E)$ and $\ell\in \N,$ suppose the problem (\ref{SDP}) has $(G,\ell)-$SPLR structure and its optimal solution has rank at least $r.$ For any $\sigma,\hat{n}\in \N$ such that $\hat{n}\geq n+\sigma,$ consider the following new SDP problem:
\begin{multline}\label{extSDP}
\min\Big\{ \<A_0,X_{1:n,1:n}\>:\ \forall i\in [m],\ \<A_i,X_{1:n,1:n}\>=b_i,\ X_{n+1:n+\sigma,1:n}={\bf 0}_{\sigma\times n},\\
 X_{1:n,n+1:n+\sigma}={\bf 0}_{n\times\sigma},\ X_{n+1:n+\sigma,n+1:n+\sigma}=I_\sigma, X\in \S^{\hat{n}}_+\Big\}.
\end{multline}
Define the graph
\begin{equation}\label{extenGlem}
\widehat{G}:=\([\hat{n}],\E\sqcup E\([n+\sigma]\setminus [n],[n]\)\sqcup E\( [n+\sigma]\setminus [n]\)\).
\end{equation}
Then problem (\ref{extSDP}) has $\big(\widehat{G},\ell\big)-$SPLR structure such that 
$\pw\big(\widehat{G}\big)=\pw\(G\)+\sigma,$ $\tw\big(\widehat{G}\big)=\tw\(G\)+\sigma$ and its optimal solution has rank at least $r+\sigma.$
\end{lem}

\begin{proof}
Because the problem (\ref{SDP}) has $(G,\ell)-$SPLR structure. It has decomposition $A_i=A_i^s+A_i^l$ such that $\dim\({\rm Col}([A_0^l,A_1^l,\ldots,A_m^l])\)\leq \ell$ and for any $i\in \setm,$ $A_i^s\in \S^n(\E,0).$ From (\ref{extenGlem}), problem (\ref{extSDP}) has 
$\big(\widehat{G},\ell\big)-$SPLR structure by considering the last three linear constraints as sparse constraints. Moreover, from the edge set of $\widehat{G}$ in (\ref{extenGlem}), it is easy to see that $\pw\big(\widehat{G}\big)=\pw(G)+\sigma$ and $\tw\big(\widehat{G}\big)=\tw(G)+\sigma.$ From the structure of (\ref{extSDP}), we can see that for any optimal solution $X\in \S^{\hat{n}}_+,$ $X_{1:n,1:n}$ is an optimal solution of (\ref{SDP}). It has rank at least $r.$ Thus, $X_{1:n+\sigma,1:n+\sigma}$ has rank at least $r+\sigma.$
\end{proof}

\section{$\varphi(\ell)=\ell+1$}\label{app-phi-exact}

\begin{prop}\label{prop-phi-exact}
For any $\ell\in\N^+,$ the function $\varphi$ defined in Definition~\ref{cvxnb} satisfies $\varphi(\ell)=\ell+1.$
\end{prop}

\begin{proof}
By Proposition~\ref{propphi}, it suffices to prove that $\varphi(\ell)\leq\ell+1.$ Let $U=[U_1;U_2]\in\R^{2\ell\times\ell}$ with $U_1,U_2\in\R^{\ell\times\ell},$ and suppose that the set (\ref{cvxset}) is nonempty. Choose an element $Y_0$ of this set and define
\[
P(t):=\begin{bmatrix}U_1\\tU_2\end{bmatrix}^{\top}Y_0
\begin{bmatrix}U_1\\tU_2\end{bmatrix}\succeq0,\qquad t\in\R.
\]
By Hanselka and Sinn~\cite[Corollary~3.8]{hanselka2019positive}, there exists $Q(t)\in\R[t]^{(\ell+1)\times\ell}$ such that $P(t)=Q(t)^\top Q(t).$ Since $P$ has degree at most two and each diagonal entry of $Q(t)^\top Q(t)$ is a sum of squares, $Q$ has degree at most one. Thus,
\[
P(t)=(E_1+tE_2)^\top(E_1+tE_2),\qquad E_1,E_2\in\R^{(\ell+1)\times\ell}.
\]
Comparing the constant and quadratic coefficients gives $E_i^\top E_i=U_i^\top U_i$ for $i\in[2].$ Hence the map $U_ix\mapsto E_ix$ is a well-defined isometry on the column space of $U_i.$ By completing orthonormal bases, it extends to an isometry from $\R^\ell$ to $\R^{\ell+1},$ whose matrix $V_i\in\R^{(\ell+1)\times\ell}$ satisfies
\[
V_i^\top V_i=I_\ell,\qquad V_iU_i=E_i,\qquad i\in[2].
\]
Now set
\[
Y_*:=[V_1,V_2]^\top[V_1,V_2]
=\begin{bmatrix}I_\ell&V_1^\top V_2\\V_2^\top V_1&I_\ell\end{bmatrix}.
\]
Then $Y_*\succeq0,$ $\rr(Y_*)\leq\ell+1,$ and
\[
U^\top Y_*U=(E_1+E_2)^\top(E_1+E_2)=P(1)=U^\top Y_0U=M.
\]
Therefore, $Y_*$ belongs to the set (\ref{cvxset}), proving $\varphi(\ell)\leq\ell+1.$
\end{proof}

\bibliographystyle{abbrv}
\bibliography{SPLR}

\end{document}